\documentclass[11pt]{article}
\usepackage[top=30truemm,bottom=30truemm,left=24truemm,right=24truemm]{geometry}

\usepackage{amsmath,amssymb,amsthm}
\usepackage{cases}
\usepackage{url}

\allowdisplaybreaks[4]

\newtheorem{dfn}{Definition}[section]
\newtheorem{lem}[dfn]{Lemma}

\newtheorem{thm}[dfn]{Theorem}
\newtheorem{cor}[dfn]{Corollary}
\newtheorem{rem}[dfn]{Remark}

\newcommand{\bff}{{\bold f}}
\newcommand{\bfg}{{\bold g}}

\newcommand{\dv}{{\rm div}\,}

\newcommand{\BR}{{\Bbb R}}
\newcommand{\BC}{{\Bbb C}}

\newcommand{\BN}{{\Bbb N}}

\newcommand{\tf}{\tilde{f}}
\newcommand{\tF}{\tilde{F}}

\renewcommand{\th}{\tilde{h}}
\newcommand{\tH}{\tilde{H}}
\newcommand{\barh}{\bar{h}}
\newcommand{\barH}{\bar{H}}

\newcommand{\tA}{\tilde{A}}

\newcommand{\hu}{\hat{u}}
\newcommand{\hv}{\hat{v}}

\newcommand{\hh}{\hat{h}}
\newcommand{\hpi}{\hat{\pi}}

\newcommand{\hW}{\hat{W}}

\newcommand{\CA}{{\mathcal A}}

\newcommand{\CD}{{\mathcal D}}

\newcommand{\CF}{{\mathcal F}}

\newcommand{\CK}{{\mathcal K}}
\newcommand{\CL}{{\mathcal L}}
\newcommand{\CM}{{\mathcal M}}

\newcommand{\CR}{{\mathcal R}}
\newcommand{\CS}{{\mathcal S}}
\newcommand{\CT}{{\mathcal T}}
\newcommand{\CH}{{\mathcal H}}

\newcommand{\pd}{\partial}

\newcommand{\HS}{{\BR^N_+}}
\newcommand{\bHS}{{\BR^N_0}}
\newcommand{\dBR}{{\dot{\BR}}}

\newcommand{\eps}{\varepsilon}
\newcommand{\loc}{\mathrm{loc}}

\renewcommand\Re{\operatorname{Re}}
\renewcommand\Im{\operatorname{Im}}

\begin{document}

\title{Maximal $L_p$-$L_q$ regularity for the Stokes equations with various boundary conditions in the half space}

\author{Naoto Kajiwara \footnote{Applied Physics Course, Department of Electrical, Electronic and Computer Engineering, Gifu University, Yanagido 1-1, Gifu, Gifu 501-1193, JAPAN. E-mail: kajiwara@gifu-u.ac.jp}}

\date{}

\maketitle


\begin{abstract}
We prove resolvent $L_p$ estimates and maximal $L_p$-$L_q$ regularity estimates for the Stokes equations with Dirichlet, Neumann and Robin boundary conditions in the half space. 
Each solution is constructed by a Fourier multiplier of $x'$-direction and an integral of $x_N$-direction. 
We decompose the solution such that the symbols of the Fourier multipliers are bounded and holomorphic. 
We see that the operator norms are dominated by a homogeneous function of order $-1$ for $x_N$-direction.
The basis are Weis's operator-valued Fourier multiplier theorem and a boundedness of a kernel operator. 
We give a new simple approach to get maximal regularity in the half space. 
\end{abstract}


\begin{flushleft}
\end{flushleft}




\section{Introduction}

This paper is concerned with resolvent $L_p$ estimates and maximal $L_p$-$L_q$ regularity for the Stokes equations with three types of boundary conditions in the half-space with $1<p,q<\infty$. 
The boundary conditions are Dirichlet(so called non-slip), Neumann and Robin(so called Navier-slip). 
The resolvent estimate is used for the generation of analytic semigroups, and the maximal regularity is used to solve quasi-linear evolution equations such as free boundary problems of non-linear Navier--Stokes equations. 
The standard way to solve the free boundary problems is to transform of a fixed domain. 
The typical methods are known as Lagrangian transform or Hanzawa transform. 
Then the equations on a fixed domain often become quasi-linear equations. 
The maximal regularity for linearized equations is useful for this solvability. 
Let $\Omega\subset \BR^N$ be a domain with three disjoint boundaries $\Gamma_D, \Gamma_N$ and $\Gamma_R$, where $N\ge 2$. 
We allow that one or two of them are empty. 
We keep in mind the following transformed linearized Stokes problem; 
\begin{equation*}\label{Stokes}\left\{\begin{aligned}
\pd_t u -  \Delta u + \nabla \pi = f  &\quad \text{in}~\Omega\times (0, \infty), \\
\dv u = g &\quad \text{in}~\Omega\times (0, \infty), \\
u=h_D&\quad \text{on}~\Gamma_D\times (0,\infty), \\
(D(u)-\pi I) n = h_N&\quad \text{on}~\Gamma_N\times (0,\infty), \\
\alpha u + \beta ( D(u) n - \langle D(u) n, n \rangle n ) = h_R, \quad \langle u, n\rangle =h_{Rn}&\quad \text{on}~\Gamma_R\times (0,\infty), \\
u|_{t=0}= u_0&\quad \text{in}~\Omega.\\
\end{aligned}\right.
\end{equation*}
Here unknowns are the velocity $u={}^t(u_1, \ldots, u_N)$ and the pressure $\pi$, while $f, g, h_D, h_N, h_R, h_{Rn}$ and $u_0$ are given functions, $n$ denotes the unit normal outward vector, $D(u)$ and $I$ are $N$ times $N$ matrices whose $(j,k)$ components $D(u)_{jk}$ and $\delta_{jk}$ are given by 
\[ D(u)_{jk}=\frac{\pd u_j}{\pd x_k} + \frac{\pd u_k}{\pd x_j}, \quad \delta_{jk}=\begin{cases}1~(j=k), \\ 0~(j\neq k), \end{cases} \]
and $\alpha\ge0, \beta>0$. 
Note that the end-point case $(\alpha, \beta)=(1,0)$ in Robin boundary condition implies Dirichlet boundary condition. 
The inner product $\langle \cdot, \cdot \rangle$ is the standard one in $\BR^N$. 
For the sake of simplicity, we set the viscosity coefficient and density equal $1$. 
Not only this non-stationary Stokes equations but also the following generalized resolvent problem are analyzed; 
\begin{equation*}\label{resolvent Stokes}\left\{\begin{aligned}
\lambda u -  \Delta u + \nabla \pi = f  &\quad \text{in}~\Omega, \\
\dv u = g &\quad \text{in}~\Omega, \\
u=h_D&\quad \text{on}~\Gamma_D, \\
(D(u)-\pi I) n = h_N&\quad \text{on}~\Gamma_N, \\
\alpha u + \beta ( D(u) n - \langle D(u) n, n \rangle n ) = h_R, \quad \langle u, n\rangle =h_{Rn}&\quad \text{on}~\Gamma_R. 
\end{aligned}\right.
\end{equation*}
This resolvent equation is derived from Laplace transform of the equation \eqref{Stokes}. 

In this paper we do not treat the domain with curved boundaries so that the domain is the half-space. 
However the domain will be allowed more general domains like a bounded domain by cut-off techniques and localizations. 
We do not use such procedures since that is common and the analysis of the half space is the most important steps. 
Instead of them, we consider the problem with non-homogeneous data, which is a key to treat non-linear problems. 
After a reduction to $f=g=0$, we consider the solution operator from boundary data $h$ to the solutions $u$ and $\pi$. 
Although these solutions are given by a Fourier multiplier of $h(x',0)$ which is independent of $x_N$-variable, we shall use $h(x', x_N)$ by using an integral. 
We decompose the symbols of the solution operators into new symbols and new independent variables. 
Since the new symbol of the Fourier multiplier operator are bounded and holomorphic, we are able to use Fourier multiplier theorem with the connection to Mikhlin conditions. 
The theorem is used only for the whole space, but we are able to use the theorem in the tangential direction for the half-space $\HS=\BR^{N-1}\times (0,\infty)$. 
We confirm that the operator norm is dominated by a homogeneous function of order $-1$ in $x_N$. 
This kernel operator is known as a bounded operator in $L_p$ spaces. 
Therefore this shows resolvent estimates. 
Note that the new decomposed independent variables become suitable right-hand side of the generalized resolvent estimates. 
Moreover we see that this symbols are $\CR$-bounded in $\lambda$. 
Therefore we are able to use Weis's operator-valued Fourier multiplier theorem to get maximal $L_p$-$L_q$ regularity. 
There are a lot of technical ideas to get the maximal regularity in the half-space. 
However we emphasize that we do not need such elaborate calculations. 
We only use the classical tools developed by Shibata et al. \cite{KS12, S20, SS03, SS08, SS12} and a book \cite{PS16} covering various subjects to harmonic analysis and the maximal regularity. 

Let us review a brief history about maximal regularity for the Stokes equations. 
The most classical result was given by Solonnikov \cite{S77} who proved maximal $L_p$-$L_p$ regularity for the Stokes equations by using Potential theory. 
Improved results to anisotropic $L_p$-$L_q$ type were given by Giga et al. in \cite{G85, GS91}. 
They used bounded imaginary powers of Stokes operators and refined Dore--Venni theorem for both bounded and exterior domains. 
In 2001, a sufficient condition for $L_p(\BR, X)$-boundedness of Fourier multiplier operators was constructed by Weis \cite{W01} in terms of $\CR$-bounded of the symbols under $X$ is $\mathcal{H}\mathcal{T}$ space. 
This breakthrough led a lot of results for the maximal regularity. 
For example, see the monographs by Denk--Hieber--Pr\"uss \cite{DHP03} and Kunstmann--Weis \cite{KW04}. 
These were applied to the elliptic operators. 
Weis's theorem was applied not only elliptic operators but also Stokes operator. 
It has shown by Geissert, Hech, Hieber and Sawada \cite{GHHS12} that the existence of the Helmholtz decomposition implies the analyticity and maximal $L_p$-$L_q$ regularity for the Stokes operators. 
Moreover we note that Farwig, Kozono and Sohr \cite{FKS05, FKS09} proved maximal $L_p$-$\tilde{L}_q$ regularity for general domains. 
A general explanation for the Stokes equations was given by \cite{HS18}. 
We heavily depend on the results by Shibata et al. \cite{KS12, SS12}. 
It was also important for them to use the theorem due to Weis, where the methods seemed systematic ways in the sense that they got the resolvent estimate and the maximal regularity at the same time. 
Since then, there are a lot of results, e.g. for model problems with Neumann or free boundary conditions \cite{SS08, SS09, SS12}, Robin conditions \cite{SS07, S07}, two-phase problems \cite{SS11}. 
For the case of general domains, see \cite{S16, S14, S20}. 
On the other hand, our method will show easier than them since the basis is bounded and holomorphic although essential ideas are similar. 
At last, see \cite{PS16} for the comprehensive results about analyticity of semigroups, vector-valued harmonic analysis, maximal regularity, parabolic and Stokes equations and its applications to the free boundary problems. 
Almost all of our main theorems have already proved before, but we give a new simple approach to get resolvent estimates and maximal regularity estimates. 
We treat Dirichlet, Neumann and Robin boundary conditions. 
We remark that the normal component of Dirichlet and Robin boundary conditions are inhomogeneous, which is a generalization of previous works. 
We introduce a little bit small function space for $h_N$. 
Similar function space were considered in the paper \cite{SS11} for two-phase Stokes problems. 
Recently, resolvent estimates and maximal regularity estimate for two-phase Stokes problem has shown in \cite{K} by the method used in this paper. 

The structure of the paper is as follows. 
First we introduce some notations and state our main theorems in section \ref{main}. 
Then, in section \ref{preliminaries}, we prepare some known definitions and theorems. 
Since the equations are inhomogeneous, we transform the equation into homogeneous except for boundary data $h$. 
This is as usual and is stated in section \ref{reduction}. 
In section \ref{formula}, we solve the equations in the half space by partial Fourier transforms. 
Three types of boundary conditions are treated similarly. 
The solution formula is Fourier multiplier type with the symbols of sum of heat part $e^{-\sqrt{\lambda + |\xi'|^2} x_N}$ and Stokes part $\CM_\lambda(\xi', x_N)$ which is defined later. 
From so called Volevich's trick, the solutions are given by an integral form whose integrands are Fourier multiplier operators which act $h$ and $\pd_N h$. 
In the last section \ref{proof}, we prove the main theorem after giving a sufficient condition to get estimates. 
We decompose the symbols while paying attention to the desired estimates. 
Resolvent estimate is straightforward from the theorem prepared in section \ref{preliminaries} and the estimates of $e^{-\sqrt{\lambda + |\xi'|^2} x_N}$ and $\CM_\lambda(\xi', x_N)$. 
Maximal regularity estimates are also same as resolvent estimates since the symbols are $\CR$-bounded in $\lambda$-variables. 
Proofs of some lemmas are written in Appendices. 

\section{Main theorem}\label{main}
We formulate the resolvent and the non-stationary problems in the half-space. 
Let $\HS$ and $\bHS$ be the half-space and its flat boundary and let $Q_+$ and $Q_0$ be the corresponding time-space domain; 
\begin{alignat*}{3}
&\HS:=\{x=(x_1, \ldots, x_N)\in\BR^N\mid x_N>0\}, &\quad \bHS&:=\{x=(x',0)=(x_1, \ldots,x_{N-1}, 0)\in\BR^N\}, \\
&Q_+:=\HS\times (0,\infty), &\quad Q_0&:=\bHS\times (0,\infty). 
\end{alignat*}

The resolvent problem is as follows; 
\begin{equation}\label{resolvent Stokes}\left\{\begin{aligned}
\lambda u -  \Delta u + \nabla \pi = f  &\quad \text{in}~\HS, \\
\dv u = g&\quad \text{in}~\HS, 
\end{aligned}\right.
\end{equation}
with one of the following boundary conditions on $\bHS$; 
\begin{alignat*}{3}
&\text{(Dirichlet)}~u = h, \\
&\text{(Neumann)}\begin{cases}-(\pd_N u_j + \pd_j u_N) = h_j\quad(j=1,\ldots, N-1),\\ -(2 \pd_N u_N - \pi) = h_N,\end{cases}\\
&\text{(Robin)}\begin{cases} \alpha u_j - \beta \pd_N u_j = h_j \quad(j=1,\ldots, N-1),\\u_N = h_N. \\\end{cases}
\end{alignat*}

The non-stationary problem is as follows; 
\begin{equation}\label{non-stationary Stokes}\left\{\begin{aligned}
\pd_t U -  \Delta U + \nabla \Pi = F  &\quad \text{in}~Q_+, \\
\dv U = G&\quad \text{in}~Q_+, 
\end{aligned}\right.
\end{equation}
with initial data $U|_{t=0}=U_0$ and with one of the following boundary conditions on $Q_0$; 
\begin{alignat*}{3}
&\text{(Dirichlet)}~U= H, \\
&\text{(Neumann)}\begin{cases}-(\pd_N U_j + \pd_j U_N) = H_j\quad(j=1,\ldots, N-1),\\ -(2 \pd_N U_N - \Pi) = H_N,\end{cases}\\
&\text{(Robin)}\begin{cases} \alpha U_j - \beta \pd_N U_j = H_j \quad(j=1,\ldots, N-1),\\U_N = H_N. \\\end{cases}
\end{alignat*}

Given a domain $D$, Lebesgue and Sobolev spaces are denoted by $L_q(D)$ and $W^m_q(D)$ with norms $\|\cdot\|_{L_q(D)}$ and $\|\cdot\|_{W^m_q(D)}$. 
Same manner is applied in the $X$-valued spaces $L_p(\BR, X)$ and $W^m_p(\BR, X)$. 
For a scalar function $f$ and $N$-vector $\bff=(f_1, \ldots, f_N)$, we use the following symbols; 
\begin{alignat*}{3}
&\nabla f = (\pd_1 f, \ldots, \pd_N f), &\quad &\nabla^2 f = (\pd_i\pd_j f \mid i,j=1,\ldots, N), \\
&\nabla \bff = (\pd_i f_j \mid i,j=1, \ldots, N), &\quad &\nabla^2 \bff = (\pd_i \pd_j f_k \mid i,j,k= 1, \ldots, N). 
\end{alignat*}
Even though $\bfg=(g_1, \ldots, g_{\tilde{N}}) \in X^{\tilde{N}}$ for some $\tilde{N}$, we use the notations $\bfg\in X$ and $\|\bfg\|_X$ as $\sum_{j=1}^{\tilde{N}}\|g_j\|_X$ for simplicity. 
Set 
\begin{align*}
\hW^1_q(D)= \{\pi\in L_{q,\loc}(D)\mid\nabla \pi\in L_q(D)\}, \quad \hW^1_{q,0}(D)=\{\pi\in \hW^1_q(D)\mid \pi|_{\pd D}=0\}
\end{align*}
and let $\hW^{-1}_q(D)$  denote the dual space of $\hW^1_{q',0}(D)$, where $1/q+1/q'=1$. 
For $\pi\in \hW^{-1}_q (D)\cap L_q(D)$, we have 
\[\|\pi\|_{\hW^{-1}_q(D)} = \sup\left\{ \left|\int_D \pi \phi dx\right| \mid \phi \in \hW^1_{q',0}(D), \|\nabla \phi\|_{L_{q'}(D)}=1\right\}. \]

Let $\CF$ and $\CF^{-1}$ denote the Fourier transform and its inverse; 
\begin{align*}
\CF[f](\xi):=\int_{\BR^n} e^{-i x\cdot \xi} f(x) dx, \quad \CF^{-1} [g](x):=\frac{1}{(2\pi)^n}\int_{\BR^n} e^{i x\cdot \xi} g(\xi) d\xi.
\end{align*}

Although we consider time interval $\BR_+$, we regard functions on $\BR$ to use Fourier transform. 
To do so and to consider Laplace transforms as Fourier transforms, we introduce some function spaces; 
\begin{align*}
L_{p,0,\gamma_0}(\BR, X)&:=\{f:\BR\to X \mid e^{-\gamma_0 t}f(t)\in L_p(\BR, X),~f(t)=0~\text{for}~t<0\}, \\
W^m_{p,0,\gamma_0}(\BR, X)&:=\{f\in L_{p,0,\gamma_0}(\BR, X) \mid e^{-\gamma_0 t}\pd_t^j f(t) \in L_p(\BR, X),~j=1,\ldots, m\}, \\
L_{p,0}(\BR, X)&:=L_{p,0,0}(\BR; X),\quad W^m_{p,0}(\BR, X):=W^m_{p,0,0}(\BR; X)
\end{align*}
for some $\gamma_0\ge0$. 
Let $\CL$ and $\CL^{-1}_\lambda$ denote two-sided Laplace transform and its inverse, defined as 
\begin{align*}
\CL[f](\lambda) = \int_{-\infty}^\infty e^{-\lambda t} f(t) dt, \quad \CL^{-1}_\lambda[g](t) = \frac{1}{2\pi}\int_{-\infty}^\infty e^{\lambda t} g(\lambda) d\tau, 
\end{align*}
where $\lambda = \gamma + i\tau \in \BC$. 
Given $s\in \BR$ and $X$-valued function $f$, we use the following Bessel potential spaces to treat fractional orders; 
\begin{align*}
H^s_{p,0,\gamma_0}(\BR, X) &:= \{f:\BR\to X \mid \Lambda^s_\gamma f := \CL^{-1}_\lambda[|\lambda|^s \CL[f](\lambda)](t) \in L_{p, 0, \gamma}(\BR, X)~\text{for~any}~\gamma\ge \gamma_0\}, \\
H^s_{p,0}(\BR, X)&:=H^s_{p,0,0}(\BR, X), \\
H^s_{p,0}(0,T, X)&:= \{f\mid {\rm there~exists}~F\in H^s_{p,0}(\BR, X)~{\rm such~that}~F|_{(0,T)}=f\},~0<T<\infty, \\
\|\Lambda_\gamma^s f\|_{L_p(0,T, X)}&:=\inf\{\|\Lambda_\gamma^s F\|_{L_p(\BR, X)}\mid F\in H^s_{p,0}(\BR, X)~{\rm such~that}~F|_{(0,T)}=f\}, f\in H^s_{p,0}(0,T, X), 
\end{align*}
where $F|_{(0,T)}$ means a restriction to $(0,T)$. 
Since we need to take care of the $N$-th component of the velocity for Dirichlet and Robin boundary conditions, we introduce the following function spaces; 
\begin{align*}
E_q(\HS)&:=\{h_N \in W^2_q(\HS)\mid |\nabla'|^{-1}\pd_N h_N := \CF_{\xi'}^{-1} |\xi'|^{-1} \CF_{x'} (\pd_N h_N)(x', x_N) \in L_q(\HS)\}. 
\end{align*}

Let $\Sigma_\eps:=\{\lambda\in\BC\setminus\{0\}\mid |\arg\lambda|<\pi-\eps\}$. 
Throughout this paper, the index of $h=(h_j)$ and $H=(H_j)$ runs from $j=1$ to $N$ when Dirichlet and Neumann boundaries, and we use the notations $h = (h', h_N)$ and $H=(H',H_N)$ when Robin boundary. 
We are ready to state our main results. 

\begin{thm}\label{resolventthm}
Let $0<\eps<\pi/2$ and $1<q<\infty$. 
Then for any $\lambda \in \Sigma_\eps$, 
\[ f\in L_q(\HS),\quad g\in \hW^{-1}_q(\HS)\cap W^1_q(\HS), \quad h\in \begin{cases} W^2_q(\HS),~h_N\in E_q(\HS)~\rm{Dirichlet}, \\ W^1_q(\HS)~\rm{Neumann}, \\W^1_q(\HS) \times E_q(\HS)~\rm{Robin}, 
\end{cases}  \]
problem \eqref{resolvent Stokes} admits a unique solution $(u, \pi) \in W^2_q(\HS) \times \hW^1_q(\HS)$ with the resolvent estimate; 
\begin{align*}
&\|(|\lambda|u, |\lambda|^{1/2}\nabla u, \nabla^2 u, \nabla \pi)\|_{L_q(\HS)} \\
\le& \begin{cases}
C\{\|(f, |\lambda|^{1/2} g, \nabla g, |\lambda|h, \nabla^2 h, |\lambda||\nabla'|^{-1} \pd_N h_N) \|_{L_q(\HS)} + |\lambda| \|g\|_{\hW^{-1}_q(\HS)}\}~\rm{Dirichlet},\\
C\{\|(f, |\lambda|^{1/2} g, \nabla g, |\lambda|^{1/2} h, \nabla h) \|_{L_q(\HS)} + |\lambda| \|g\|_{\hW^{-1}_q(\HS)}\}~\rm{Neumann}, \\
C'\{\|(f, |\lambda|^{1/2} g, \nabla g, |\lambda|^{1/2} h', \nabla h', |\lambda|h_N, \nabla^2 h_N, |\lambda||\nabla'|^{-1} \pd_N h_N) \|_{L_q(\HS)} \\
\qquad \qquad + \frac{\alpha}{|\lambda|^{1/2}} \|(f, |\lambda|^{1/2}g, \nabla g)\|_{L_q(\HS)} + |\lambda| \|g\|_{\hW^{-1}_q(\HS)}\}~\rm{Robin}, \end{cases}
\end{align*}
for some constants $C=C_{N, q, \eps}$ and $C'=C'_{N, q, \eps, \alpha, \beta}$. 
\end{thm}

\begin{rem}
For Robin boundary condition, the right-hand side of theorem \ref{resolventthm} can be replaced by 
\begin{align*}
&C_{N, q, \alpha, \beta, \eps, \delta} \{\|(f, |\lambda|^{1/2} g, \nabla g, |\lambda|^{1/2} h', \nabla h', |\lambda|h_N, \nabla^2 h_N, |\lambda||\nabla'|^{-1} \pd_N h_N) \|_{L_q(\HS)} + |\lambda| \|g\|_{\hW^{-1}_q(\HS)}\}, 
\end{align*}
provided $\lambda\in \Sigma_\eps$ with $|\lambda|\ge\delta>0$. 
\end{rem}

\begin{thm}\label{maxregthm}
Let $1<p, q<\infty$ and $\gamma_0>0$, or $\gamma_0\ge0$ when Dirichlet, Neumann or Robin boundary with $\alpha=0$. 
Then for any 
\begin{align*}
F&\in L_{p,0,\gamma_0}(\BR, L_q(\HS)),\quad G \in L_{p,0,\gamma_0}(\BR, W^1_q(\HS))\cap W^1_{p,0,\gamma_0}(\BR, \hW^{-1}_q(\HS)), \\
H&\in \begin{cases}
W^1_{p,0,\gamma_0}(\BR, L_q(\HS)) \cap L_{p,0,\gamma_0}(\BR, W^2_q(\HS)), H_N\in H^1_{p,0,\gamma_0}(\BR, E_q(\HS))~\rm{Dirichlet}, \\
H^{1/2}_{p,0,\gamma_0}(\BR, L_q(\HS))\cap L_{p,0,\gamma_0}(\BR, W^1_q(\HS))~\rm{Neumann},\\
\left(H^{1/2}_{p,0,\gamma_0}(\BR, L_q(\HS))\cap L_{p,0,\gamma_0}(\BR, W^1_q(\HS))\right)\\
\hspace{15mm}\times \left(W^1_{p,0,\gamma_0}(\BR, L_q(\HS)) \cap L_{p,0,\gamma_0}(\BR, W^2_q(\HS)) \cap H^1_{p,0, \gamma_0}(\BR, E_q(\HS))\right)~\rm{Robin},
\end{cases}
\end{align*}
problem \eqref{non-stationary Stokes} with $U_0=0$ admits a unique solution $(U, \Pi)$ such that 
\begin{align*}
U&\in W^1_{p,0,\gamma_0}(\BR, L_q(\HS)) \cap L_{p,0,\gamma_0}(\BR, W^2_q(\HS)), \\
\Pi &\in L_{p,0,\gamma_0}(\BR, \hW^1_q(\HS))
\end{align*} 
with the maximal $L_p$-$L_q$ regularity; 
\begin{align*}
&\|e^{-\gamma t}(\pd_t U, \gamma U, \Lambda^{1/2}_\gamma \nabla U, \nabla^2 U, \nabla \Pi)\|_{L_p(\BR, L_q(\HS))} \\
\le& \begin{cases}
C\{\|e^{-\gamma t} (F, \Lambda^{1/2}_\gamma G, \nabla G, \pd_t H, \nabla^2 H, \pd_t(|\nabla'|^{-1}\pd_N H_N) \|_{L_p(\BR, L_q(\HS))} \\
\hspace{15mm}+ \|e^{-\gamma t} (\pd_t G, \gamma G)\|_{L_p(\BR, \hW^{-1}_q(\HS))}\}~\rm{Dirichlet},\\
C\{\|e^{-\gamma t} (F, \Lambda^{1/2}_\gamma G, \nabla G, \Lambda^{1/2}_\gamma H, \nabla H) \|_{L_p(\BR, L_q(\HS))} + \|e^{-\gamma t} (\pd_t G, \gamma G)\|_{L_p(\BR, \hW^{-1}_q(\HS))} \}~\rm{Neumann}, \\
C'\{\|e^{-\gamma t} (F, \Lambda^{1/2}_\gamma G, \nabla G, \Lambda^{1/2}_\gamma H', \nabla H', \pd_t H_N, \nabla^2 H_N, \pd_t(|\nabla'|^{-1}\pd_N H_N) \|_{L_p(\BR, L_q(\HS))} \\
\hspace{15mm} + \|e^{-\gamma t} (\pd_t G, \gamma G)\|_{L_p(\BR, \hW^{-1}_q(\HS))}\}~\rm{Robin},
\end{cases}
\end{align*}
for any $\gamma \ge \gamma_0$ with some constants $C=C_{N, p, q, \gamma_0}$ and $C'=C'_{N, p, q,\gamma_0, \alpha, \beta}$. 
\end{thm}

Above two theorems are our main theorems. 
We write down other applications. 
We introduce some function spaces and Stokes operator on $L_q(\HS)$; 
\begin{align*}
&J_q(\HS):=\begin{cases}\{u\in L_q(\HS)\mid \dv u = 0~{\rm in}~\HS,~u_N=0~{\rm on}~\BR^N_0\}~{\rm Dirichlet, Robin}, \\
\{u\in L_q(\HS)\mid \dv u = 0~{\rm in}~\HS\}~{\rm Neumann}, 
\end{cases}\\
&G_q(\HS)=\begin{cases}\{\nabla \pi \mid \pi\in \hW^1_q(\HS)\}~{\rm Dirichlet, Robin}, \\
\{\nabla \pi \mid \pi\in W^1_{q,0}(\HS)\}~{\rm Neumann}, 
\end{cases}\\
&D(A_q):=\begin{cases}\{u\in J_q(\HS)\cap W^2_q(\HS)\mid u_j=0,~j=1, \ldots, N-1,~{\rm on}~\BR^N_0\}~{\rm Dirichlet}, \\
\left\{u\in J_q(\HS)\cap W^2_q(\HS)\;\middle|\; \begin{aligned}&\pd_N u_j + \pd_j u_N = 0,~j=1, \ldots, N-1, \\ & \pd_N u_N  = 0\end{aligned}~{\rm on}~\BR^N_0\right\}~{\rm Neumann}, \\
\{u\in J_q(\HS)\cap W^2_q(\HS)\mid \alpha u_j - \beta \pd_N u_j = 0,~j=1, \ldots, N-1, {\rm on}~\BR^N_0\}~{\rm Robin}, 
\end{cases}\\
&A_q u := - P_q \Delta u, 
\end{align*}
where $P_q$ is the continuous projection from $L_q(\HS)$ onto $J_q(\HS)$ along $G_q(\HS)$ corresponding to Helmholtz decomposition $L_q(\HS)=J_q(\HS)\oplus G_q(\HS)$. 
More precisely, see \cite{G11, SS08, S07}. 
Here we did not write $(L_q(\HS))^N$ and  $(W^2_q(\HS))^N$ as mentioned above. 
Theorem \ref{resolventthm} with $g=h=0$ implies the generation of analytic semigroups on $J_q(\HS)$. 
\begin{cor}
Let $1<q<\infty$. 
For Dirichlet, Neumann or Robin boundary conditions, the Stokes operator $-A_q$ generates a bounded analytic semigroup $\{e^{-A_q t}\}_{t\ge0}$ on $J_q(\HS)$. 
\end{cor}

Let $\CD_{q,p}(\HS) = [J_q(\HS), D(A_q)]_{1-1/p,p}$ denote the real interpolation space. 
It is characterised as follows; 
\begin{align*}
&\CD_{q,p}(\HS)=\\
&\begin{cases}
J_q(\HS)\cap B^{2(1-1/p)}_{q,p}(\HS)~{\rm if}~\frac{2}{p}+\frac{1}{q}>2, {\rm Dirichlet}, \\
J_q(\HS)\cap B^{2(1-1/p)}_{q,p}(\HS)~{\rm if}~\frac{2}{p}+\frac{1}{q}>1, {\rm Neumann, Robin}, \\
\{u\in J_q(\HS)\cap B^{2(1-1/p)}_{q,p}(\HS)\mid u_j=0,~j=1, \ldots, N-1,~{\rm on}~\BR^N_0\}~{\rm if}~\frac{2}{p}+\frac{1}{q}<2, {\rm Dirichlet}, \\
\left\{u\in J_q(\HS)\cap B^{2(1-1/p)}_{q,p}(\HS)\;\middle|\; \begin{aligned}&\pd_N u_j + \pd_j u_N = 0,~j=1, \ldots, N-1, \\ & \pd_N u_N = 0\end{aligned}~{\rm on}~\BR^N_0\right\}~{\rm if}~\frac{2}{p}+\frac{1}{q}<1,~{\rm Neumann}, \\
\{u\in J_q(\HS)\cap B^{2(1-1/p)}_{q,p}(\HS)\mid \alpha u_j - \beta \pd_N u_j = 0,~j=1, \ldots, N-1, ~{\rm on}~\BR^N_0\}~{\rm if}~\frac{2}{p}+\frac{1}{q}<1,~{\rm Robin}. 
\end{cases}
\end{align*}

The maximal regularity on finite time interval is a corollary of our theorems (cf, \cite{SS08, S07}). 
\begin{thm}
Let $1<p,q<\infty$ and $0<T<\infty$. 
Then for any 
\begin{align*}
F&\in L_p(0,T, L_q(\HS)),\quad G \in L_p(0,T, W^1_q(\HS))\cap W^1_{p,0}(0,T, \hW^{-1}_q(\HS)), \\
H&\in \begin{cases}
W^1_{p,0}(0,T, L_q(\HS)) \cap L_p(0,T, W^2_q(\HS))~,H_N\in H^1_{p,0}(0,T, E_q(\HS))~\rm{Dirichlet}, \\
H^{1/2}_{p,0}(0,T, L_q(\HS))\cap L_p(0,T, W^1_q(\HS))~\rm{Neumann},\\
\left(H^{1/2}_{p,0}(0,T, L_q(\HS))\cap L_p (0,T, W^1_q(\HS))\right)\times \\
\hspace{15mm} \left(W^1_{p,0}(\BR, L_q(\HS)) \cap L_p(0,T, W^2_q(\HS))\cap H^1_{p,0}(0,T, E_q(\HS))\right)~\rm{Robin},
\end{cases}\\
U_0&\in \CD_{q,p}(\HS), 
\end{align*}
problem \eqref{non-stationary Stokes} admits a unique solution $(U, \Pi)$ such that 
\begin{align*}
U&\in W^1_p(0,T, L_q(\HS)) \cap L_p(0,T, W^2_q(\HS)), \\
\Pi &\in L_p(0,T, \hW^1_q(\HS))
\end{align*} 
with the maximal $L_p$-$L_q$ regularity; for all $\gamma \ge \gamma_0$, 
\begin{align*}
&\|(\pd_t U, \gamma U, \Lambda^{1/2}_\gamma \nabla U, \nabla^2 U, \nabla \Pi)\|_{L_p(0,T, L_q(\HS))} \\
\le& \begin{cases}
C\{\|(F, \Lambda^{1/2}_\gamma G, \nabla G, \pd_t H, \nabla^2 H, \pd_t(|\nabla'|^{-1}\pd_N H_N) \|_{L_p(0,T, L_q(\HS))}  \\
\qquad  + \|(\pd_t G, \gamma G)\|_{L_p(0,T, \hW^{-1}_q(\HS))} + \|U_0\|_{B^{2(1-1/p)}_{q,p}(\HS)}\}~\rm{Dirichlet},\\
C\{\|(F, \Lambda^{1/2}_\gamma G, \nabla G, \Lambda^{1/2}_\gamma H, \nabla H) \|_{L_p(0,T, L_q(\HS))} \\
\qquad + \|(\pd_t G, \gamma G)\|_{L_p(0,T, \hW^{-1}_q(\HS))} + \|U_0\|_{B^{2(1-1/p)}_{q,p}(\HS)}\}~\rm{Neumann}, \\
C'\{\|(F, \Lambda^{1/2}_\gamma G, \nabla G, \Lambda^{1/2}_\gamma H', \nabla H', \pd_t H_N, \nabla^2 H_N, \pd_t(|\nabla'|^{-1}\pd_N H_N) ) \|_{L_p(0,T, L_q(\HS))} \\
\qquad + \|(\pd_t G, \gamma G)\|_{L_p(0,T, \hW^{-1}_q(\HS))} + \|U_0\|_{B^{2(1-1/p)}_{q,p}(\HS)}\}~\rm{Robin}, 
\end{cases}
\end{align*}
with some constants $C=C_{N, p, q, \gamma_0, T}$ and $C'=C'_{N, p, q, \gamma_0, T, \alpha, \beta}$, where $\gamma_0>0$, or $\gamma_0\ge0$ when Dirichlet, Neumann and Robin boundary with $\alpha=0$. 
\end{thm}

\begin{rem}
In fact, the proposition holds for $T=\infty$ for Dirichlet or Robin boundary conditions. 
These cases correspond to the exponential decay of $e^{-tA_q} U_0$ because $0$ belongs to the resolvent set of Stokes operator $A_q$. 
For Neumann boundary case, the paper \cite{SS08} treated global results by considering a quotient space. 
\end{rem}

\begin{rem}
As far as we know, almost all results are $h_N=0$ and $H_N=0$ since that is physically reasonable. 
However we are able to extend non-homogeneous boundary data. 
\end{rem}

\section{Preliminaries}\label{preliminaries}

Since key lemma in this paper is operator-valued Fourier multiplier theorem, we need some preparations for the base spaces and the symbols. 
Almost all of results in this section can be found in the book \cite{PS16}. 

\begin{dfn}
A Banach space $X$ is said to belong to the class $\CH\CT$ if the Hilbert transform $\CH$, defined by 
\[ \CH f:= \frac{1}{\pi} \lim_{\eps\to+0} \int_{|t-s|>\eps} \frac{f(s)}{t-s} ds, \]
is a bounded linear operator on $L_p(\BR, X)$ for some $1<p<\infty$.  
In this case we write $X\in \CH\CT$. 
\end{dfn}

Let $X$ and $Y$ be Banach spaces with norms $\|\cdot\|_X$ and $\|\cdot \|_Y$. 
Let $\CL(X, Y)$ denote the set of all bounded linear operators from $X$ to $Y$, and $\CL(X):=\CL(X,X)$. 

\begin{dfn}
A family of operators $\CT\subset \CL(X, Y)$ is called $\CR$-bounded, if there exist constant $C>0$ and $1\le p < \infty$ such that for each $m \in \BN, \{T_j\}_{j=1}^m \subset \CT, \{x_j\}_{j=1}^m \subset X$ and for all sequences $\{\eps_j(u)\}_{j=1}^m$ of independent, symmetric, $\{-1, 1\}$-valued random variables on a probability space $(\Omega, \CA, \mu)$ the inequality 
\[ |\sum_{j=1}^m \eps_j T_j x_j |_{L_p(\Omega, Y)} \le C  |\sum_{j=1}^m \eps_j x_j|_{L_p(\Omega, X)}\]
is valid. 
The smallest such $C$ is called $\CR$-bound of $\CT$, which is denoted by $\CR(\CT)$. 
\end{dfn}
Note that when $X$ and $Y$ are Hilbert spaces, $\CT\subset\CL(X,Y)$ is $\CR$-bounded if and only if $\CT$ is uniformly bounded. 

Let $\dBR:=\BR\setminus\{0\}$ and $\dBR^n=[\dBR]^n$. 
Given $M\in C (\dBR^n, \CL(X, Y))$, we define an operator $T_M : \CF^{-1}\CD(\BR^n, X) \to \CS(\BR^n, Y)$ by means of 
\[ T_M \phi := \CF^{-1} M \CF \phi, \quad \text{for~all~}\CF\phi\in \CD(\BR^n, X). \] 

\begin{thm}[Weis \cite{W01}]
Let $X$, $Y\in \CH\CT$ and $1<p<\infty$. 
Let $M\in C^1(\dBR, \CL(X, Y))$ satisfy 
\[\CR(\{ (\xi \frac{d}{d\xi})^j M(\xi) \mid \xi \in \dBR, j=0,1\}) = \kappa < \infty. \]
Then the operator $T_M$ is a bounded linear operator from $L_p(\BR, X)$ to $L_p(\BR, Y)$. 
Moreover 
\[\|T_M\|_{\CL(L_p(\BR, X), L_p(\BR, Y))}\le C\kappa\]
for some positive constant $C$ depending on $X$, $Y$ and $p$. 
\end{thm}

\begin{dfn}
A Banach space $X$ is said to have property $(\alpha)$ if there exists a constant $\alpha>0$ such that 
\[|\sum_{i,j=1}^m \alpha_{ij} \eps_i \eps_j' x_{ij}|_{L_2(\Omega\times \Omega', X)} \le \alpha |\sum_{i,j=1}^m \eps_i \eps_j' x_{ij}|_{L_2(\Omega\times \Omega', X)}, \]
for all $\alpha_{ij}\in\{-1,1\}, x_{ij}\in X, m\in \BN$, and all symmetric independent $\{-1, 1\}$-valued random variables $\{\eps_i\}_{i=1}^m$ resp. $\{\eps_j'\}_{j=1}^m$ on a probability space $(\Omega, \CA, \mu)$ resp. $(\Omega', \CA', \mu')$. 
The class $\CH\CT(\alpha)$ denotes the set of all Banach spaces which belong to $\CH\CT$ and have property $(\alpha)$. 
\end{dfn}

\begin{rem}
For any Hilbert space $E$, we have $E\in \CH\CT(\alpha)$. 
If $(S, \Sigma, \sigma)$ is a sigma-finite measure space and $1<p<\infty$, then $L_p(S, E) \in \CH\CT(\alpha)$ as well. 
\end{rem}

\begin{thm}\label{Rbound}
Let $1<p<\infty$, $X$, $Y\in \CH\CT(\alpha)$, and suppose that the family of multipliers $\CM\subset C^n(\dBR^n, \CL(X, Y))$ satisfies 
\[\CR(\{ \xi^\alpha \pd_\xi^\alpha M(\xi) \mid \xi\in \dBR^n, \alpha\in\{0,1\}^n, M\in \CM\})=:\kappa<\infty.\]
Then the family of operators $\CT:=\{T_M\mid M\in\CM\} \subset \CL(L_p(\BR^n, X), L_p(\BR^n, Y))$ is $\CR$-bounded with $\CR(\CT)\le C\kappa$, where $C>0$ only depends on $X$, $Y$ and $p$. 
\end{thm}

Let $\tilde{\Sigma}_\eta:=\{z\in\BC\setminus\{0\}\mid |\arg z|<\eta\} \cup \{z\in\BC\setminus\{0\}\mid \pi-\eta<|\arg z|\}$ for $\eta\in(0, \pi/2)$. 
To verify the Lizorkin condition in above theorem, a useful sufficient condition is known in terms of holomorphic and boundedness, which is denoted by the class $H^\infty$. 

\begin{thm}[{\cite[Proposition 4.3.10]{PS16}}]\label{H^infty}
Let $X, Y$ be Banach spaces and suppose that, for some $0<\eta<\pi/2$, the family of multipliers $\CM\subset H^\infty(\tilde{\Sigma}_\eta^n, \CL(X, Y))$ satisfies 
\[\CR(\{M(z)\mid z\in \tilde{\Sigma}_\eta^n, M\in \CM\})=:\kappa < \infty. \]
Then 
\[\CR(\{\xi^\alpha \pd_\xi^\alpha M(\xi)\mid \xi\in\dBR^n, |\alpha|=k, M\in \CM\})\le \kappa/(\sin\eta)^k, \]
for each $k\in \BN_0$. 
\end{thm}

From above theorem, we do not need to show $\CR$-boundedness of the derivatives when multipliers are bounded and holomorphic. 
To take over $\CR$-boundedness, we need a dominated theorem below. 

\begin{thm}[{\cite[Proposition 4.1.5]{PS16}}]\label{Rkernel}
Let $X$, $Y$ be Banach spaces, $D\subset \BR^n$, and $1<p<\infty$. 
Suppose $\CK\subset \CL(L_p(D, X), L_p(D, Y))$ is a family of kernel operators in the sense that 
\[Kf(x) = \int_D k(x,x')f(x') dx', \quad x\in D, f\in L_p(D, X), \]
for each $K\in \CK$, where the kernels $k:D\times D\to \CL(X,Y)$ are measurable, with 
\[\CR(\{k(x,x'): K\in \CK\}) \le k_0(x,x'), \quad x,x'\in D, \]
and the operator $K_0$ with scalar kernel $k_0$ is bounded in $L_p(D)$. 
Then $\CK\subset \CL(L_p(D, X), L_p(D, Y))$ is $\CR$-bounded and $\CR(\CK)\le \|K_0\|_{L_p(D)}$. 
\end{thm}

Moreover we use the following theorem of the boundedness of a kernel operator. 

\begin{lem}[{\cite[Lemma 5.5]{SS12}, \cite[Proposition 1.4.16]{KS12}}]\label{kernel}
Let $X$ be a Banach space, $k(t,s)$ be a function defined on $(0,\infty)\times (0,\infty)$ which satisfies the condition: $k(\lambda t, \lambda s)= \lambda^{-1} k(t,s)$ for any $\lambda>0$ and $(t,s)\in (0,\infty)\times (0,\infty)$. 
In addition, we assume that for some $1\le q < \infty$ 
\[\int_0^\infty |k(1, s)|s^{-1/q} ds =: A_q < \infty. \]
If we define the integral operator $T$ by the formula: 
\[[Tf](t)=\int_0^\infty k(t,s)f(s)ds, \]
then $T$ is a bounded linear operator on $L_q(\BR_+, X)$ and 
\[ \|Tf\|_{L_q(\BR_+, X)} \le A_q \|f\|_{L_q(\BR_+, X)}.\]
\end{lem}
We use the theorem for $k(t,s)=(t+s)^{-1}$ which satisfies the assumption. 

\section{Reduction to the problem only with boundary data}\label{reduction}

In this section we show that it is enough to consider the case $f=g=0$ or $F=G=0$ by subtracting solutions of inhomogeneous data. 

\subsection{Whole space}
We start considering with the whole space problem 
\begin{alignat}{5}
&\lambda u - \Delta u + \nabla \pi =f, &\quad&\dv u = g&\quad &{\rm in}~\BR^N, \label{R1}\\
 &\pd_t U - \Delta U + \nabla \Pi = F, &\quad &\dv U=G&\quad &{\rm in}~\BR^N\times(0,\infty) \label{R2}
\end{alignat}
subject to the initial condition $U(x,0)=0$. 
The following theorem is known for the whole space. 
\begin{thm}[{\cite[Theorem 3.1]{SS12}}]
Let $1<p, q<\infty, 0<\eps<\pi/2$ and $\gamma_0\ge0$. \\
{\rm (1)} For any $\lambda\in\Sigma_\eps, f\in L_q(\BR^N), g\in\hW^{-1}_q(\BR^N)\cap W^1_q(\BR^N)$, problem \eqref{R1} admits a unique solution $(u,\pi)\in W^2_q(\BR^N)\times \hW^1_q(\BR^N)$ that satisfies the following estimates: 
\[\|(|\lambda|u, |\lambda|^{1/2}\nabla u, \nabla^2 u, \nabla \pi)\|_{L_q(\BR^N)} \\
\le C_{N,q,\eps}\{\|(f, |\lambda|^{1/2} g, \nabla g) \|_{L_q(\BR^N)} + |\lambda| \|g\|_{\hW^{-1}_q(\BR^N)}\}.\]
{\rm (2)} For any $F\in L_{p,0,\gamma_0}(\BR, L_q(\BR^N))$ and $G\in L_{p,0,\gamma_0}(\BR, W^1_q(\BR^N))\cap W^1_{p,0,\gamma_0}(\BR, \hW^{-1}_q(\BR^N))$, problem \eqref{R2} admits a unique solution 
\[(U, \Pi)\in(W^1_{p,0,\gamma_0}(\BR, L_q(\BR^N))\cap L_{p,0,\gamma_0}(\BR, W^2_q(\BR^N)))\times L_{p,0,\gamma_0}(\BR, \hW^1_q(\BR^N))\]
that satisfies the estimate: 
\begin{align*}
&\|e^{-\gamma t}(\pd_t U, \gamma U, \Lambda^{1/2}_\gamma \nabla U, \nabla^2 U, \nabla \Pi)\|_{L_p(\BR, L_q(\BR^N))} \\
&\qquad \le C_{N,p,q, \gamma_0}\{\|e^{-\gamma t} (F, \Lambda^{1/2}_\gamma G, \nabla G) \|_{L_p(\BR, L_q(\BR^N))} + \|e^{-\gamma t} (\pd_t G, \gamma G)\|_{L_p(\BR, \hW^{-1}_q(\BR^N))}\}
\end{align*}
for any $\gamma\ge\gamma_0$. 
\end{thm}

\subsection{Half space}
Concerning the half space, we first consider the divergence problem 
\begin{alignat}{5}
&\dv v = g&\quad &{\rm in}~\HS,&\quad&v_N|_{\BR^N_0}=0,~\pd_N v_j|_{\BR^N_0}=0~(j=1,\ldots, N-1),\label{R+1}\\
&\dv V = G&\quad &{\rm in}~Q_+,&\quad&V_N|_{Q_0}=0,~\pd_N V_j|_{Q_0}=0~(j=1,\ldots, N-1), \label{R+2}
\end{alignat}
subject to the initial condition $V|_{t=0}=0$. 
\begin{lem}\label{divfree}
Let $1<p,q<\infty$ and $\gamma_0\ge0$. \\
{\rm (1)} For any $g\in \hW^{-1}_q(\HS)\cap W^1_q(\HS)$, problem \eqref{R+1} admits a solution $v\in W^2_q(\HS)$ that satisfies the following estimates: 
\begin{align*}
\|v\|_{L_q(\HS)}&\le C\|g\|_{\hW^{-1}_q(\HS)}, \\
\|\nabla^{j+1} v\|_{L_q(\HS)}&\le C\|\nabla^j g\|_{L_q(\HS)}~(j=0,1). \tag{*}
\end{align*}
{\rm (2)} For any $G\in L_{p,0,\gamma_0}(\BR, W^1_q(\HS))\cap W^1_{p,0,\gamma_0}(\BR, \hW^{-1}_q(\HS))$, problem \eqref{R+2} admits a solution 
\[V\in W^1_{p,0,\gamma_0}(\BR, L_q(\HS))\cap L_{p,0,\gamma_0}(\BR, W^2_q(\HS))\] 
that satisfies the estimates: 
\begin{align*}
\|e^{-\gamma t}\pd_t V\|_{L_p(\BR, L_q(\HS))} &\le C\|e^{-\gamma t}\pd_t G\|_{L_p(\BR, \hW^{-1}_q(\HS))}, \\
\|e^{-\gamma t}\Lambda^{1/2}_\gamma V\|_{L_p(\BR, L_q(\HS))}&\le C\|e^{-\gamma t}\Lambda^{-1/2}_\gamma \pd_t G\|_{L_p(\BR, \hW^{-1}_q(\HS))}, \\
\|e^{-\gamma t} \nabla V\|_{L_p(\BR, L_q(\HS))} &\le C\|e^{-\gamma t} G\|_{L_p(\BR, L_q(\HS))}, \tag{*}\\
\|e^{-\gamma t}\Lambda^{1/2}_\gamma \nabla V\|_{L_p(\BR, L_q(\HS))}&\le C\|e^{-\gamma t}\Lambda^{1/2}_\gamma G\|_{L_p(\BR, L_q(\HS))}, \tag{*}\\
\|e^{-\gamma t}\nabla^2 V\|_{L_p(\BR, L_q(\HS))} &\le C\|e^{-\gamma t}\nabla G\|_{L_p(\BR, L_q(\HS))}\tag{*}
\end{align*}
for any $\gamma\ge\gamma_0$. 
If $\gamma\ge\gamma_0>0$, we have 
\begin{align*} 
\|e^{-\gamma t} \Lambda_\gamma^{-1/2}\pd_t G \|_{L_p(\BR, \hW^{-1}_q(\HS))} &\le C\|e^{-\gamma t} \pd_t G \|_{L_p(\BR, \hW^{-1}_q(\HS))}, \\
\|e^{-\gamma t} G\|_{L_p(\BR, L_q(\HS))} &\le C\|e^{-\gamma t} \Lambda_\gamma^{1/2}G\|_{L_p(\BR, L_q(\HS))}. 
\end{align*}
\end{lem}

\begin{proof}
Throughout this paper, let $f^e$ and $F^e$ be even extensions to $\BR^N$ and let $f^o$ and $F^o$ be odd extensions to $\BR^N$; 
 \begin{alignat*}{3}
 &f^e(x):=\begin{cases}f(x)&{\rm for}~x_N>0\\f(x',-x_N)&{\rm for}~x_N<0\end{cases}, &\qquad &F^e(x, t):=\begin{cases}F(x, t)&{\rm for}~x_N>0\\F(x',-x_N, t)&{\rm for}~x_N<0\end{cases}, \\
&f^o(x):=\begin{cases}f(x)&{\rm for}~x_N>0\\-f(x',-x_N)&{\rm for}~x_N<0\end{cases}, &\qquad &F^e(x,t):=\begin{cases}F(x, t)&{\rm for}~x_N>0\\-F(x',-x_N, t)&{\rm for}~x_N<0. \end{cases}
\end{alignat*} 
Then 
\begin{alignat*}{3}
&v_j(x)=\CF^{-1}_\xi \left[\frac{-i\xi_j \CF_x[g^e](\xi)}{|\xi|^2}\right](x),&\qquad &v(x)=(v_1,\ldots,v_N)(x), \\
&V_j(x, t)=\CF^{-1}_\xi \left[\frac{-i\xi_j \CF_x[G^e](\xi, t)}{|\xi|^2}\right](x,t),&\qquad &V(t,x)=(V_1,\ldots,V_N)(t,x)
\end{alignat*}
are solutions. 
In fact, the same proof in \cite[Lem 4.1]{SS12} works and they satisfy for the estimates (*). 
We need to prove boundary conditions and other estimates. 

To prove boundary conditions, we calculate as follows. 
\begin{align*}
v_N(x',0) &= \frac{1}{(2\pi)^N} \int_{\BR^{N-1}}e^{ix'\cdot\xi'}\left(\int_{-\infty}^\infty \frac{-i\xi_N \CF_x[g^e](\xi)}{|\xi|^2}d\xi_N\right)d\xi'\\
\CF_{x'}[v_N|_{x_N=0}](\xi') &= \frac{-i}{(2\pi)^N} \int_{-\infty}^\infty \frac{\xi_N}{|\xi|^2}\CF_x[g^e](\xi)d\xi_N\\
&=\frac{-i}{(2\pi)^N} \int_{-\infty}^\infty \frac{\xi_N}{|\xi|^2}\left(\int_{-\infty}^\infty e^{-iy_N\xi_N}[\CF_{x'}g^e](\xi',y_N)dy_N\right)d\xi_N\\
&=\frac{-i}{(2\pi)^N} \int_{-\infty}^\infty \frac{\xi_N}{|\xi|^2}\left(\int_0^\infty (e^{-iy_N\xi_N}+e^{iy_N\xi_N})[\CF_{x'}g](\xi',y_N)dy_N\right)d\xi_N\\
&=\frac{-i}{(2\pi)^N} \int_0^\infty \left(\int_{-\infty}^\infty \frac{\xi_N}{|\xi|^2}(e^{-iy_N\xi_N}+e^{iy_N\xi_N})d\xi_N\right)[\CF_{x'}g](\xi',y_N)dy_N\\
&=0
\end{align*}
since $\xi_N\mapsto\frac{\xi_N}{|\xi|^2}(e^{-iy_N\xi_N}+e^{iy_N\xi_N})$ is an odd function. 
Therefore $v_N|_{\BR^N_0}=0$. 
Another boundary condition $\pd_N v_j|_{\BR^N_0}=0$ is also proved similarly. 
To estimate $\|v\|_{L_q(\HS)}$, we divide into two cases; $j\neq N$ and $j=N$. 
If $j\neq N$, then, from even property of $v_j$ and $\CF_\xi^{-1} (\frac{-i\xi_j\CF_x [\phi^e](\xi)}{|\xi|^2})$ on $x_N$, for all $\phi\in C^\infty_0(\HS)$, 
\[(v_j, \phi)_{\HS} = \frac{1}{2}(v_j, \phi^e)_{\BR^N} = \frac{1}{2}\left(g^e, \CF_\xi^{-1} \left[\frac{-i\xi_j\CF_x [\phi^e](\xi)}{|\xi|^2}\right]\right)_{\BR^N} = \left(g, \CF_\xi^{-1} \left[\frac{-i\xi_j\CF_x [\phi^e](\xi)}{|\xi|^2}\right]\right)_{\HS}, \]
where $(\cdot, \cdot)$ is the standard $L^2$-inner product, and therefore
\[|(v_j, \phi)_{\HS}|\le \|g\|_{\hW^{-1}_q(\HS)} \|\nabla \CF_\xi^{-1} \left[\frac{-i\xi_j\CF_x [\phi^e](\xi)}{|\xi|^2}\right]\|_{L_{q'}(\HS)} \le C\|g\|_{\hW^{-1}_q(\HS)} \|\phi\|_{L_{q'}(\HS)}.\]
On the other hand, if $j=N$, then, from odd property of $v_N$ and even property of $\CF_\xi^{-1} (\frac{-i\xi_N\CF_x [\phi^o](\xi)}{|\xi|^2})$ on $x_N$, for all $\phi\in C^\infty_0(\HS)$, 
\[(v_N, \phi)_{\HS} = \frac{1}{2}(v_N, \phi^o)_{\BR^N} = \frac{1}{2}\left(g^e, \CF_\xi^{-1} \left[\frac{-i\xi_N\CF_x [\phi^o](\xi)}{|\xi|^2}\right]\right)_{\BR^N} = \left(g, \CF_\xi^{-1} \left[\frac{-i\xi_N\CF_x [\phi^o](\xi)}{|\xi|^2}\right]\right)_{\HS}, \]
and therefore
\[|(v_N, \phi)_{\HS}|\le \|g\|_{\hW^{-1}_q(\HS)} \|\nabla \CF_\xi^{-1} \left[\frac{-i\xi_N\CF_x [\phi^o](\xi)}{|\xi|^2}\right]\|_{L_{q'}(\HS)} \le C\|g\|_{\hW^{-1}_q(\HS)} \|\phi\|_{L_{q'}(\HS)}, \]
which means the desired estimate. 
By same argument, we get $V_N|_{Q_0}=\pd_N V_j|_{Q_0}=0$ and the estimate for $V$. 
For the remained estimates, we use $\Lambda_\gamma^{-1/2}$ is a bounded linear operator when $\gamma\ge\gamma_0>0$. 
\end{proof}

Setting $u=v+w, \tf=f-(\lambda v-\Delta v)$ and $U=V+W, \tF=F-(\pd_t V -\Delta V)$, we would like to find $(w,\pi), (W, \Pi)$ such that 
\begin{equation}\label{resolvent Stokes g=0}\left\{\begin{aligned}
\lambda w -  \Delta w + \nabla \pi = \tf  &\quad \text{in}~\HS, \\
\dv w = 0&\quad \text{in}~\HS, 
\end{aligned}\right.
\end{equation}
with 
\begin{alignat*}{3}
&\text{(Dirichlet)}\begin{cases} w_j = h_j-v_j=:\th_j^D\quad(j=1,\ldots, N-1),\\w_N=h_N=:\th_N^D,\end{cases}\\
&\text{(Neumann)}\begin{cases}-(\pd_N w_j + \pd_j w_N) = h_j+ \pd_j v_N=:\th_j^N\quad(j=1,\ldots, N-1),\\ -(2 \pd_N w_N - \pi) = h_N + 2\pd_N v_N=:\th_N^N ,\end{cases}\\
&\text{(Robin)}\begin{cases} \alpha w_j - \beta \pd_N w_j = h_j  - \alpha v_j=:\th_j^R\quad(j=1,\ldots, N-1),\\w_N = h_N=:\th_N^R, \\\end{cases}
\end{alignat*} 
and 
\begin{equation}\label{non-stationary Stokes G=0}\left\{\begin{aligned}
\pd_t W-  \Delta W + \nabla \Pi = \tF  &\quad \text{in}~Q_+, \\
\dv W = 0&\quad \text{in}~Q_+, \\
W|_{t=0}=0 &, 
\end{aligned}\right.
\end{equation}
with 
\begin{alignat*}{3}
&\text{(Dirichlet)}\begin{cases} W_j = H_j-V_j=:\tH_j^D\quad(j=1,\ldots, N-1),\\W_N=H_N=:\tH_N^D,\end{cases}\\
&\text{(Neumann)}\begin{cases}-(\pd_N W_j + \pd_j W_N) = H_j + \pd_j V_N=:\tH_j^N\quad(j=1,\ldots, N-1),\\ -(2 \pd_N W_N - \Pi) = H_N + 2\pd_N V_N=:\tH_N^N,\end{cases}\\
&\text{(Robin)}\begin{cases} \alpha W_j - \beta \pd_N W_j = H_j - \alpha V_j=:\tH_j^R\quad(j=1,\ldots, N-1),\\W_N = H_N=:\tH_N^R. \\\end{cases}
\end{alignat*}
By Lemma \ref{divfree}, we have 
\begin{align*}
\|\tf\|_{L_q(\HS)}&\le \|f\|_{L_q(\HS)} + C(|\lambda|\|g\|_{\hW^{-1}_q(\HS)} + \|\nabla g\|_{L_q(\HS)})\\
\|e^{-\gamma t}\tF\|_{L_p(\BR, L_q(\HS))}&\le \|e^{-\gamma t} F\|_{L_p(\BR, L_q(\HS))} + C(\|e^{-\gamma t}\pd_t G \|_{L_p(\BR, \hW^{-1}_q(\HS))} + \|e^{-\gamma t} \nabla G\|_{L_p(\BR, L_q(\HS))}), 
\end{align*}
and for Dirichlet boundary condition, 
\begin{align*}
&\|(|\lambda|\th^D, |\lambda|^{1/2}\nabla \th^D, \nabla^2 \th^D, |\lambda||\nabla'|^{-1}\pd_N \th_N)\|_{L_q(\HS)}\\ 
&\le \|(|\lambda|h, |\lambda|^{1/2}\nabla h, \nabla^2 h, |\lambda||\nabla'|^{-1}\pd_N h_N)\|_{L_q(\HS)} + C(\|(|\lambda|^{1/2}g, \nabla g)\|_{L_q(\HS)} + |\lambda|\|g\|_{\hW^{-1}_q(\HS)}), \\
&\|e^{-\gamma t}(\pd_t \tH^D, \Lambda_\gamma^{1/2}\nabla \tH^D, \nabla^2 \tH^D, \pd_t (|\nabla'|^{-1}\pd_N \tH_N))\|_{L_p(\BR, L_q(\HS))}\\
&\le \|e^{-\gamma t}(\pd_t H, \Lambda_\gamma^{1/2}\nabla H, \nabla^2 H, \pd_t (|\nabla'|^{-1}\pd_N H_N))\|_{L_p(\BR, L_q(\HS))} \\
&\qquad \qquad + C(\|e^{-\gamma t}(\Lambda_\gamma^{1/2}G, \nabla G)\|_{L_p(\BR, L_q(\HS))}+ \|e^{-\gamma t}\pd_t G\|_{L_p(\BR, \hW^{-1}_q(\HS))}), 
\end{align*}
and for Neumann boundary condition, 
\begin{align*}
\|(|\lambda|^{1/2}\th^N, \nabla \th^N)\|_{L_q(\HS)} &\le \|(|\lambda|^{1/2}h, \nabla h)\|_{L_q(\HS)} + C\|(|\lambda|^{1/2}g, \nabla g)\|_{L_q(\HS)}, \\
\|e^{-\gamma t}(\Lambda_\gamma^{1/2}\tH^N, \nabla \tH^N)\|_{L_p(\BR, L_q(\HS))}&\le \|e^{-\gamma t} (\Lambda_\gamma^{1/2}H, \nabla H)\|_{L_p(\BR, L_q(\HS))} + C\|e^{-\gamma t}(\Lambda_\gamma^{1/2}G, \nabla G)\|_{L_p(\BR, L_q(\HS))}, 
\end{align*}    
and for Robin boundary condition, 
\begin{align*}
&\|(|\lambda|^{1/2}\th^{'R}, \nabla \th^{'R}, |\lambda|\th_N^R, \nabla^2 \th_N^R, |\lambda||\nabla'|^{-1}\pd_N \th_N^R)\|_{L_q(\HS)} \\
&\le \|(|\lambda|^{1/2}h^{'R}, \nabla h^{'R}, |\lambda|h_N^R, \nabla^2 h_N^R, |\lambda||\nabla'|^{-1}\pd_N h_N^R)\|_{L_q(\HS)} + C \frac{\alpha}{|\lambda|^{1/2}}(\||\lambda|^{1/2} g\|_{L_q(\HS)} + |\lambda|\|g\|_{\hW^{-1}_q(\HS)}), \\
&\|e^{-\gamma t}(\Lambda_\gamma^{1/2}\tH^{'R}, \nabla \tH^{'R}, \pd_t \tH_N^R, \nabla^2 \tH_N^R, \pd_t (|\nabla'|^{-1}\pd_N \tH_N^R))\|_{L_p(\BR, L_q(\HS))}\\
&\le \|e^{-\gamma t} (\Lambda_\gamma^{1/2}H', \nabla H', \pd_t H_N^R, \nabla^2 H_N^R, \pd_t (|\nabla'|^{-1} \pd_N H_N^R))\|_{L_p(\BR, L_q(\HS))} \\
&\qquad \qquad + \alpha C(\|e^{-\gamma t} G\|_{L_p(\BR, L_q(\HS))}+ \|e^{-\gamma t}\Lambda_\gamma^{-1/2} \pd_t G \|_{L_p(\BR, \hW^{-1}_q(\HS))}) \\
&\hspace{-3mm}\overset{{\rm if} \gamma_0>0}{\le} \|e^{-\gamma t} (\Lambda_\gamma^{1/2}H', \nabla H', \pd_t H_N^R, \nabla^2 H_N^R, \pd_t (|\nabla'|^{-1}\pd_N H_N^R))\|_{L_p(\BR, L_q(\HS))} \\
&\qquad \qquad + C(\|e^{-\gamma t}\Lambda_\gamma^{1/2}G\|_{L_p(\BR, L_q(\HS))} + \|e^{-\gamma t}\pd_t G \|_{L_p(\BR, \hW^{-1}_q(\HS))}). 
\end{align*}
From above observations, we consider the case $g=0$, $G=0$ below. 

Second, we reduce the case $f=0$, $F=0$. 
For given $f\in L_q(\HS)$, $F\in L_{p,0, \gamma_0}(\BR, L_q(\HS))$, we define $\iota f:=(f_1^o, \ldots, f_{N-1}^o, f_N^e)$, $\iota F:=(F_1^o, \ldots, F_{N-1}^o, F_N^e)$ which are the functions on the whole space. 
Let $P(\xi)=(P_{j,k})_{jk}=(\delta_{jk}-\xi_j\xi_k|\xi|^{-2})_{jk}$ be the Helmholtz decomposition. 
Then functions 
\begin{alignat*}{3}
v(x)&=\CF^{-1}_\xi \left[\frac{P(\xi) \CF_x [\iota f](\xi)}{\lambda+|\xi|^2}\right](x), &\quad
\tau(x)&= -\CF^{-1}_\xi\left[\frac{i\xi\cdot\CF_x [\iota f](\xi)}{|\xi|^2}\right](x), \\
V(x,t)&= \CL_\lambda \CF^{-1}_\xi \left[\frac{P(\xi) \CF_x \CL[\iota F](\xi, \lambda)}{\lambda+|\xi|^2}\right](x,t), &\quad 
\Upsilon(x,t)&= -\CL_\lambda\CF^{-1}_\xi\left[\frac{i\xi\cdot\CF_x\CL[\iota F](\xi, \lambda)}{|\xi|^2}\right](x,t)
\end{alignat*}
satisfy 
\begin{align*}
&(v, \tau)\in W^2_q(\BR^N)\times \hW^1_q(\BR^N), \\
&\lambda v - \Delta v + \nabla \tau = \iota f,\quad \dv v=0 \quad {\rm in}~\BR^N, \\
&\|(|\lambda|v, |\lambda|^{1/2}\nabla v, \nabla^2 v, \nabla \tau)\|_{L_q(\BR^N)} \le C\|\iota f\|_{L_q(\BR^N)} \le 2C\|f\|_{L_q(\HS)}
\end{align*}
and 
\begin{align*}
&V\in W^1_{p,0, \gamma_0}(\BR, L_q(\BR^N))\cap L_{p,0, \gamma_0}(\BR, W^2_q(\BR^N)),\quad \Upsilon\in L_{p,0, \gamma_0}(\BR, \hW^1_q(\BR^N)), \\
&\pd_t V  - \Delta V + \nabla \Upsilon= \iota F,\quad \dv V=0 \quad {\rm in}~\BR^N\times (0,\infty),\quad V|_{t=0}=0, \\
&\|e^{-\gamma t}(\pd_t V, \gamma V, \Lambda_\gamma^{1/2} \nabla V, \nabla^2 V, \nabla \Upsilon)\|_{L_p(\BR, L_q(\BR^N))} \le C\|e^{-\gamma t} \iota F\|_{L_p(\BR, L_q(\BR^N))} \le 2C\|e^{-\gamma t}F\|_{L_p(\BR, L_q(\HS))}, 
\end{align*}
for any $1<p,q<\infty$, $\gamma\ge\gamma_0\ge0$ and $\lambda\in\Sigma_\eps$ with $0<\eps<\pi/2$. 
Moreover we shall prove  $v_j(x',0)=0$ and $V_j(x',0,t)=0$ for $j=1, \ldots, N-1$ as follows; 
\begin{align*}
&(2\pi)^N \CF_{x'}[v_j|_{x_N=0}](\xi') \\
&= \sum_{k=1, k\neq j}^{N-1}\int_{-\infty}^\infty \frac{P_{j,k}(\xi)\CF_x[f^o_k](\xi)}{\lambda +|\xi|^2} d\xi_N + \int_{-\infty}^\infty \frac{P_{j,j}(\xi)\CF_x[f^o_j](\xi)}{\lambda+|\xi|^2} d\xi_N +  \int_{-\infty}^\infty \frac{P_{j,N}(\xi)\CF_x[f^e_N](\xi)}{\lambda+|\xi|^2} d\xi_N \\
&=\sum_{k=1, k\neq j}^{N-1}\int_{-\infty}^\infty \frac{-\xi_j \xi_k}{(\lambda+|\xi|^2)|\xi|^2}\left(\int_0^\infty (e^{-iy_N\xi_N}- e^{iy_N\xi_N})[\CF_{x'}f_k](\xi',y_N)dy_N\right)d\xi_N \\
&\qquad + \int_{-\infty}^\infty \frac{1-\xi_j^2 |\xi|^{-2}}{\lambda+|\xi|^2}\left(\int_0^\infty (e^{-iy_N\xi_N}- e^{iy_N\xi_N})[\CF_{x'}f_j](\xi',y_N)dy_N\right)d\xi_N \\
&\qquad + \int_{-\infty}^\infty \frac{-\xi_j \xi_N}{(\lambda+|\xi|^2)|\xi|^2}\left(\int_0^\infty (e^{-iy_N\xi_N}+e^{iy_N\xi_N})[\CF_{x'}f_N](\xi',y_N)dy_N\right)d\xi_N\\
&=-\sum_{k=1, k\neq j}^{N-1}\int_0^\infty \left(\int_{-\infty}^\infty \frac{e^{-iy_N\xi_N}- e^{iy_N\xi_N}}{(\lambda+|\xi|^2)|\xi|^2}d\xi_N\right) \xi_j \xi_k[\CF_{x'}f_k](\xi',y_N)dy_N\\
&\qquad + \int_0^\infty \left(\int_{-\infty}^\infty \frac{e^{-iy_N\xi_N}- e^{iy_N\xi_N}}{\lambda+|\xi|^2}d\xi_N\right) [\CF_{x'}f_j](\xi',y_N)dy_N\\
&\qquad - \int_0^\infty \left(\int_{-\infty}^\infty \frac{e^{-iy_N\xi_N}- e^{iy_N\xi_N}}{(\lambda+|\xi|^2)|\xi|^2}d\xi_N\right) \xi_j^2[\CF_{x'}f_j](\xi',y_N)dy_N\\
&\qquad - \int_0^\infty \left(\int_{-\infty}^\infty \frac{\xi_N (e^{-iy_N\xi_N}+e^{iy_N\xi_N})}{(\lambda+|\xi|^2)|\xi|^2}d\xi_N\right) \xi_j[\CF_{x'}f_N](\xi',y_N)dy_N\\
&=0, 
\end{align*}
where we used odd property on $\xi_N$ for all integrals. 
We also see $\pd_N v_N=\tau=0$ and $\pd_N V_N=\Upsilon=0$ on the boundary in the same manner. 
See also in \cite[p.399]{SS03}. 

Setting $u=v+w$, $\pi=\tau+\kappa$ in \eqref{resolvent Stokes} with $g=0$ and $U=V+W$, $\Pi=\Upsilon+\Xi$ in \eqref{non-stationary Stokes} with $G=0$ and $U_0=0$, respectively, we have 
\begin{equation}\label{resolvent Stokes f=g=0}\left\{\begin{aligned}
\lambda w -  \Delta w + \nabla \kappa = 0  &\quad \text{in}~\HS, \\
\dv w = 0&\quad \text{in}~\HS, 
\end{aligned}\right.
\end{equation}
with 
\begin{alignat*}{3}
&\text{(Dirichlet)}~w= h-v=:\barh^D, \\
&\text{(Neumann)}\begin{cases}-(\pd_N w_j + \pd_j w_N) = h_j+(\pd_N v_j + \pd_j v_N)=:\barh_j^N\quad(j=1,\ldots, N-1),\\ -(2 \pd_N w_N - \kappa) = h_N=:\barh_N^N ,\end{cases}\\
&\text{(Robin)}\begin{cases} \alpha w_j - \beta \pd_N w_j = h_j  + \beta \pd_N v_j=:\barh_j^R\quad(j=1,\ldots, N-1),\\w_N = h_N - v_N=:\barh_N^R \\\end{cases}
\end{alignat*} 
and 
\begin{equation}\label{non-stationary Stokes F=G=0}\left\{\begin{aligned}
\pd_t W-  \Delta W + \nabla \Xi = 0  &\quad \text{in}~Q_+, \\
\dv W = 0&\quad \text{in}~Q_+, \\
W|_{t=0}=0 &, 
\end{aligned}\right.
\end{equation}
with 
\begin{alignat*}{3}
&\text{(Dirichlet)}~W= H-V=:\barH^D, \\
&\text{(Neumann)}\begin{cases}-(\pd_N W_j + \pd_j W_N) = H_j + (\pd_N V_j + \pd_j V_N)=:\barH_j^N\quad(j=1,\ldots, N-1),\\ -(2 \pd_N W_N - \Xi) = H_N =:\barH_N^N,\end{cases}\\
&\text{(Robin)}\begin{cases} \alpha W_j - \beta \pd_N W_j = H_j + \beta \pd_N V_j=:\barH_j^R\quad(j=1,\ldots, N-1),\\W_N = H_N - V_N =:\barH^R_N. \\\end{cases}
\end{alignat*}
Here we have 
\begin{align*}
&\|(|\lambda|\barh^D, \nabla^2 \barh^D, |\lambda||\nabla'|^{-1}\pd_N\barh_N^D)\|_{L_q(\HS)}\le \|(|\lambda|h, \nabla^2 h, |\lambda||\nabla'|^{-1}\pd_N h_N^D)\|_{L_q(\HS)} + C\|f\|_{L_q(\HS)}, \\
&\|e^{-\gamma t}(\pd_t \barH^D, \nabla^2 \barH^D, \pd_t(|\nabla'|^{-1}\pd_N\barH_N^D))\|_{L_p(\BR, L_q(\HS))}\\
&\qquad \qquad \le \|e^{-\gamma t}(\pd_t H, \nabla^2 H, \pd_t(|\nabla'|^{-1}\pd_N H_N^D))\|_{L_p(\BR, L_q(\HS))}+ C\|e^{-\gamma t}F\|_{L_p(\BR, L_q(\HS))}, \\
&\|(|\lambda|^{1/2}\barh^N, \nabla \barh^N)\|_{L_q(\HS)}\le \|(|\lambda|^{1/2} h, \nabla h)\|_{L_q(\HS)} + C\|f\|_{L_q(\HS)}, \\
&\|e^{-\gamma t}(\Lambda_\gamma^{1/2}\barH^N, \nabla \barH^N)\|_{L_p(\BR, L_q(\HS))}\le \|e^{-\gamma t}(\Lambda_\gamma^{1/2} H, \nabla H)\|_{L_p(\BR, L_q(\HS))} + C\|e^{-\gamma t}F\|_{L_p(\BR, L_q(\HS))},  \\
&\|(|\lambda|^{1/2}\barh'^R, \nabla \barh'^R, |\lambda|\barh_N^R, \nabla^2 \barh_N^R, |\lambda||\nabla'|^{-1}\pd_N\barh_N^R)\|_{L_q(\HS)}\\
&\qquad \qquad \le \|(|\lambda|^{1/2}h', \nabla h', |\lambda|h_N, \nabla^2 h_N, |\lambda|(|\nabla'|^{-1}\pd_N h_N^R)\|_{L_q(\HS)} + C \|f\|_{L_q(\HS)}, \\
&\|e^{-\gamma t}(\Lambda_\gamma^{1/2}\barH'^R, \nabla \barH'^R, \pd_t \barH_N, \nabla^2\barH_N, \pd_t (|\nabla'|^{-1}\pd_N \barH_N^R)\|_{L_p(\BR, L_q(\HS))}\\
&\le \|e^{-\gamma t}(\Lambda_\gamma^{1/2}H', \nabla H', \pd_t H_N, \nabla^2 H_N, \pd_t (|\nabla'|^{-1}\pd_N H_N^R)\|_{L_p(\BR, L_q(\HS))}+ C\|e^{-\gamma t}F\|_{L_p(\BR, L_q(\HS))}, 
\end{align*}
where we set $\barh'^R=(\barh_j^R)_{j=1}^{N-1}$, $\barH'^R=(\barH_j^R)_{j=1}^{N-1}$ in Robin boundary. 
We used the following estimates; 
\begin{align*}
\||\lambda|(|\nabla'|^{-1}\pd_N v_N)\|_{L_q(\HS)} &\le C\|f\|_{L_q(\HS)}, \\ 
\|\pd_t (|\nabla'|^{-1}\pd_N V_N)\|_{L_p(\BR, L_q(\HS))} &\le C\|F\|_{L_p(\BR, L_q(\HS))}, 
\end{align*}
whose proof is given in \ref{appX}. 

In this section we conclude that $f=g=0$ and $F=G=0$ are enough to consider in theorems \ref{resolventthm} and \ref{maxregthm}. 

\section{Solution formulas from boundary data}\label{formula}

We give the solution of the resolvent problem \eqref{resolvent Stokes} with $f=g=0$ and $\lambda\in \Sigma_\eps$ by Fourier multipliers for each boundary condition. 
We apply partial Fourier transform with respect to tangential direction $x'\in\BR^{N-1}$ so that we use the notations 
\begin{align*}
\hv(\xi', x_N):= &\CF_{x'}v(\xi', x_N):=\int_{\BR^{N-1}} e^{-ix'\cdot\xi'} v(x', x_N)dx', \\
&\CF^{-1}_{\xi'}w(x',x_N) = \frac{1}{(2\pi)^{N-1}} \int_{\BR^{N-1}} e^{ix'\cdot\xi'} w(\xi', x_N)d\xi' 
\end{align*}
for functions $v, w:\HS\to \BC$. 
In this section and section \ref{proof} the index $j$ runs from $1$ to $N-1$ if we do not indicate. 
We use $A:=\sqrt{\sum_{j=1}^{N-1}\xi_j^2}$ and $B:= \sqrt{\lambda + A^2}$ with positive real parts. 

\subsection{Dirichlet boundary}
In this subsection we focus on Dirichlet boundary condition. 
By partial Fourier transform, we have the following second order ordinary differential equations; 
\begin{equation*}\left\{
\begin{aligned}
(\lambda + |\xi'|^2-\pd_N^2)\hu_j + i\xi_j \hpi =0\quad &\text{in}~ x_N>0, \\
(\lambda + |\xi'|^2-\pd_N^2)\hu_N + \pd_N \hpi =0\quad &\text{in}~x_N>0, \\
\sum_{j=1}^{N-1} i\xi_j\hu_j  + \pd_N \hu_N  = 0\quad &\text{in}~x_N>0, \\
\hu =\hh\quad &\text{on}~x_N=0. 
\end{aligned}\right.
\end{equation*}
We find the solution of the form 
\[\hu_j(\xi', x_N) = \alpha_j  e^{-Ax_N} + \beta_j  e^{-B x_N}~(j=1,\ldots, N), \qquad \hpi(\xi',x_N)=\gamma  e^{- A x_N}. \]
Then, the equations are 
\begin{equation*}\left\{
\begin{aligned}
\{\alpha_j(B^2 - A^2) + i\xi_j \gamma\} e^{-A x_N} = 0, \\
\{\alpha_N(B^2 - A^2)  - A\gamma \} e^{-A x_N} = 0, \\
(\sum_{j=1}^{N-1}i\alpha_j \xi_j - A\alpha_N)e^{-Ax_N} + (\sum_{j=1}^{N-1}i\beta_j\xi_j - B\beta_N)e^{-B x_N}=0, \\
\alpha_j + \beta_j=\hh_j, \quad \alpha_N + \beta_N = \hh_N. 
\end{aligned}\right.
\end{equation*}
By the linear independence of $e^{-Ax_N}$ and $e^{-Bx_N}$, we are able to find the coefficients $\alpha_j$, $\beta_j$ and $\gamma$; 
\begin{align*}
\alpha_j = \frac{i \xi_j}{A(B-A)}(i \xi'\cdot\hh' - B \hh_N), &\quad \alpha_N = - \frac{i\xi'\cdot \hh'- B \hh_N}{B-A}, \\
\beta_j = \hh_j - \alpha_j, & \quad \beta_N= \hh_N -\alpha_N, \quad \gamma = - \frac{A+B}{A}(i \xi'\cdot \hh'- B\hh_N), 
\end{align*}
where $\xi'\cdot \hh'=\sum_{k=1}^{N-1}\xi_k \hh_k$. 
We introduce the new notation 
\[\CM_\lambda(\xi', x_N) = \frac{e^{-B x_N} - e^{-Ax_N}}{B-A}\]
to treat $B-A$ in the denominator.  
Then, we have 
\begin{align*}
\hu_j(\xi',x_N) &= \sum_{k=1}^{N-1} \left\{ \delta_{jk} e^{-B x_N} + \frac{\xi_j\xi_k}{A}\CM_\lambda(\xi', x_N)  \right\} \hh_k(\xi',0) + \frac{i\xi_j B}{A}\CM_\lambda(\xi', x_N) \hh_N(\xi',0), \\
\hu_N(\xi',x_N) &= \sum_{k=1}^{N-1} \left\{ i \xi_k\CM_\lambda(\xi, x_N) \right\} \hh_k(\xi',0) + (e^{-Bx_N}- B\CM_\lambda(\xi', x_N))\hh_N(\xi',0), \\
\hpi(\xi',x_N) &= \sum_{k=1}^{N-1} \left\{-\frac{i \xi_k (A+B)}{A} e^{-A x_N}\right\} \hh_k(\xi',0) + \frac{(A+B)B}{A}e^{-Ax_N}\hh_N(\xi',0), 
\end{align*}
To simplify, we define the symbols; 
\begin{alignat*}{3}
&\phi_{j,k}^D(\lambda, \xi', x_N)=\delta_{jk} e^{-B x_N} + \frac{\xi_j\xi_k}{A}\CM_\lambda(\xi', x_N), &\quad \phi_{j,N}^D(\lambda, \xi', x_N)&=\frac{i\xi_j B}{A}\CM_\lambda(\xi', x_N), \\
&\phi_{N,k}^D(\lambda, \xi', x_N)=i \xi_k\CM_\lambda(\xi, x_N), &\quad \phi_{N,N}^D(\lambda, \xi', x_N)&=e^{-Bx_N}- B\CM_\lambda(\xi', x_N), \\
&\chi_j^D(\lambda, \xi', x_N)=- \frac{i \xi_j(A+B)}{A} e^{-A x_N}, &\quad  \chi_N^D(\lambda, \xi', x_N)&= \frac{(A+B)B}{A}e^{-Ax_N}, 
\end{alignat*}
which derives the solution formula; 
\begin{align*}
\hu_j(\xi',x_N) &=\sum_{k=1}^{N} \phi_{j,k}^D(\lambda, \xi', x_N) \hh_k(\xi',0) \quad (j=1, \ldots, N), \\
\hpi(\xi',x_N) &=\sum_{k=1}^{N} \chi_k^D(\lambda, \xi', x_N) \hh_k(\xi',0). 
\end{align*}

In the next step, we use the Volevich trick $a(\xi',0)=-\int_0^\infty \pd_N a(\xi',y_N) dy_N$ for a suitable decaying function $a$. 
We obtain the solution formula; 
\begin{align*}
u_j(x)&=- \sum_{k=1}^N \left\{\int_0^\infty  \CF_{\xi'}^{-1}\left[(\pd_N\phi_{j,k}^D(\lambda, \xi', x_N+y_N))\CF_{x'} h_k\right](x, y_N) dy_N \right.\\
&\qquad \qquad \left. + \int_0^\infty \CF_{\xi'}^{-1} \left[\phi_{j,k}^D(\lambda, \xi', x_N+y_N)\CF_{x'} (\pd_N h_k)\right](x,y_N) dy_N\right\} \quad (j=1, \ldots, N),\\
\pi(x)&=- \sum_{k=1}^N \left\{\int_0^\infty  \CF_{\xi'}^{-1}\left[(\pd_N\chi_k^D(\lambda, \xi', x_N+y_N))\CF_{x'} h_k\right](x, y_N) dy_N \right.\\
&\qquad \qquad \left. + \int_0^\infty \CF_{\xi'}^{-1} \left[\chi_k^D(\lambda, \xi', x_N+y_N)\CF_{x'} (\pd_N h_k)\right](x,y_N) dy_N\right\}. 
\end{align*}

Since Laplace transformed non-stationary Stokes equations \eqref{non-stationary Stokes} with $F=G=0$ on $\BR$ are the resolvent problem \eqref{resolvent Stokes}, we have the following formula for Dirichlet boundary condition; 
\begin{align*}
U_j(x,t)&=- \CL^{-1}_\lambda\sum_{k=1}^N \left\{\int_0^\infty  \CF_{\xi'}^{-1}\left[(\pd_N\phi_{j,k}^D(\lambda, \xi', x_N+y_N))\CF_{x'} \CL H_k\right](x, y_N) dy_N \right.\\
&\qquad \qquad \left. + \int_0^\infty \CF_{\xi'}^{-1} \left[\phi_{j,k}^D(\lambda, \xi', x_N+y_N)\CF_{x'} \CL(\pd_N H_k)\right](x,y_N) dy_N\right\} \quad (j=1, \ldots, N),\\
\Pi(x,t)&=- \CL^{-1}_\lambda\sum_{k=1}^N \left\{\int_0^\infty  \CF_{\xi'}^{-1}\left[(\pd_N\chi_k^D(\lambda, \xi', x_N+y_N))\CF_{x'} \CL H_k\right](x, y_N) dy_N \right.\\
&\qquad \qquad \left. + \int_0^\infty \CF_{\xi'}^{-1} \left[\chi_k^D(\lambda, \xi', x_N+y_N)\CF_{x'} \CL(\pd_N H_k)\right](x,y_N) dy_N\right\}. \\
\end{align*}

\subsection{Neumann boundary}
The corresponding ordinary differential equations are as follows; 
\begin{equation*}\left\{
\begin{aligned}
(\lambda + |\xi'|^2-\pd_N^2)\hu_j + i\xi_j \hpi =0\quad &\text{in}~ x_N>0, \\
(\lambda + |\xi'|^2-\pd_N^2)\hu_N + \pd_N \hpi =0\quad &\text{in}~ x_N>0, \\
\sum_{j=1}^{N-1} i\xi_j\hu_j  + \pd_N \hu_N  = 0\quad &\text{in}~ x_N>0, \\
- (\pd_N \hu_j + i \xi_j \hu_N)=\hh_j, \quad -(2\pd_N \hu_N - \hpi)=\hh_N\quad &\text{on}~ x_N=0. 
\end{aligned}\right.
\end{equation*}
The solutions are given by 
\begin{align*}
\hu_j(\xi',x_N) &= \sum_{k=1}^{N-1} \left\{(\frac{\delta_{jk}}{B} - \frac{\xi_j \xi_k(3B-A)}{D(A,B)B}) e^{- Bx_N} +  \frac{2\xi_j \xi_k B}{D(A,B)} \CM_\lambda(\xi', x_N) \right\} \hh_k(\xi',0) \\
&\qquad + \left(\frac{-i(B-A)}{D(A,B)}\xi_j e^{-B x_N} + \frac{i\xi_j(A^2+B^2)}{D(A,B)} \CM_\lambda(\xi', x_N)\right)  \hh_N(\xi',0), \\
\hu_N(\xi',x_N) &= \sum_{k=1}^{N-1} \left\{\frac{i\xi_k(B-A)}{D(A,B)} e^{- Bx_N} +  \frac{2i\xi_k AB}{D(A,B)}  \CM_\lambda(\xi', x_N) \right\} \hh_k(\xi',0)\\
&\qquad + \left(\frac{A(A+B)}{D(A,B)} e^{-B x_N} - \frac{A(A^2+B^2)}{D(A,B)} \CM_\lambda(\xi', x_N)\right)  \hh_N(\xi',0), \\
\hpi(\xi',x_N)& = \sum_{k=1}^{N-1} \left\{\frac{-2i\xi_k B(A+B)}{D(A,B)} e^{- Ax_N} \right\}\hh_k(\xi',0) +  \frac{(A+B)(A^2+B^2)}{D(A,B)}  e^{-Ax_N} \hh_N(\xi',0), 
\end{align*}
where $D(A,B)=B^3+AB^2+3A^2B-A^3$. 
It is known that $D(A,B)\neq 0$ for $\lambda\in \Sigma_\eps$, $\xi'\in \BR^{N-1}\setminus\{0\}$ in \cite[Lemma 4.4]{SS03}. 
For the details of Neumann boundary condition, see also \cite{SS12}. 
Let 
\begin{alignat*}{3}
&\phi_{j,k}^N(\lambda, \xi', x_N)=(\frac{\delta_{jk}}{B} - \frac{\xi_j \xi_k(3B-A)}{D(A,B)B}) e^{- Bx_N} +  \frac{2\xi_j \xi_k B}{D(A,B)} \CM_\lambda(\xi', x_N), \\
&\phi_{j,N}^N(\lambda, \xi', x_N)=\frac{-i\xi_j (B-A)}{D(A,B)}e^{-B x_N} + \frac{i \xi_j (A^2+B^2)}{D(A,B)}\CM_\lambda(\xi', x_N), \\
&\phi_{N,k}^N(\lambda, \xi', x_N)=\frac{i\xi_k(B-A)}{D(A,B)} e^{- Bx_N} +  \frac{2i\xi_k  AB}{D(A,B)}\CM_\lambda(\xi', x_N), \\
&\phi_{N,N}^N(\lambda, \xi', x_N)=\frac{A(A+B)}{D(A,B)} e^{-B x_N} - \frac{A(A^2+B^2)}{D(A,B)} \CM_\lambda(\xi', x_N), \\
&\chi_j^N(\lambda, \xi', x_N)=\frac{-2i\xi_j B(A+B)}{D(A,B)} e^{- Ax_N}, \quad  \chi_N^N(\lambda, \xi', x_N)= \frac{(A+B)(A^2+B^2)}{D(A,B)}  e^{-Ax_N}, 
\end{alignat*}
then, the solution formula is written as Dirichlet boundary. 

\subsection{Robin boundary}
The symbols of the solutions are given by 
\begin{alignat*}{3}
&\phi_{j,k}^R(\lambda, \xi', x_N)=(\frac{\delta_{jk}}{\alpha+ \beta B} - \frac{\beta\xi_j \xi_k}{(\alpha + \beta B)(\alpha + \beta(A+B))A}) e^{- Bx_N} + \frac{\xi_j \xi_k}{(\alpha+\beta(A+B))A}  \CM_\lambda(\xi', x_N), \\
&\phi_{j,N}^R(\lambda, \xi', x_N)=- \frac{i \beta\xi_j B }{(\alpha+\beta(A+B))A} e^{-Bx_N} + \frac{i\xi_j B(\alpha + \beta B)}{(\alpha+\beta(A+B))A}\CM_\lambda(\xi', x_N)\\
&\phi_{N,k}^R(\lambda, \xi', x_N)=\frac{i\xi_k }{\alpha + \beta(A+B)} \CM_\lambda(\xi', x_N), \\
&\phi_{N,N}^R(\lambda, \xi', x_N)= e^{-Bx_N} - \frac{B(\alpha+\beta B)}{\alpha+\beta(A+B)} \CM_\lambda(\xi', x_N), \\
&\chi_j^R(\lambda, \xi', x_N)=\frac{-i\xi_j (A+B)}{(\alpha + \beta(A+B))A}  e^{-Ax_N}, \quad \chi_N^R(\lambda, \xi', x_N)= \frac{B(A+B)(\alpha+\beta B)}{(\alpha + \beta(A+B))A} e^{-Ax_N}. 
\end{alignat*}

For the details of Robin boundary condition, see \cite{SS07, S07} although they only treated the case $h_N=0$. 

\section{Proof of resolvent estimates and maximal regularity estimates}\label{proof}

\subsection{A Sufficient condition}
We prepare a theorem to prove the main theorems. 
This gives an easy way to show a boundedness of an operator. 
Let us difine the operators $T$ and $\tilde{T}_\gamma$ by 
\begin{align*}
T[m] f(x)&=\int_0^\infty [\CF^{-1}_{\xi'} m(\xi', x_N+y_N) \CF_{x'} f](x, y_N)dy_N, \\
\tilde{T}_\gamma[m_\lambda] g(x,t)&=\CL_\lambda^{-1} \int_0^\infty [\CF^{-1}_{\xi'} m_\lambda(\xi', x_N+y_N) \CF_{x'} \CL g](x,y_N)dy_N, \\
&=[e^{\gamma t}\CF^{-1}_{\tau\to t} T[m_\lambda] \CF_{t\to\tau} (e^{-\gamma t}g)](x,t), 
\end{align*}
where $\lambda=\gamma+i\tau\in\Sigma_\eps$, $m, m_\lambda:\HS\to\BC$ are multipliers, and $f:\HS\to\BC$ and $g:\BR\times\HS\to\BC$. 

\begin{thm}\label{thm}
{\rm (i)}  
Let $m$ satisfy the following two conditions: \\
{\rm (a)} There exists $\eta\in(0,\pi/2)$ such that $\{m(\cdot, x_N), x_N>0\}\subset H^\infty(\tilde{\Sigma}_\eta^{N-1})$. \\      
{\rm (b)} There exist $\eta\in(0,\pi/2)$ and $C>0$ such that $\sup_{\xi'\in\tilde{\Sigma}^{N-1}_\eta}|m(\xi', x_N)|\le Cx_N^{-1}$ for all $x_N>0$. \\
Then $T[m]$ is a bounded linear operator on $L_q(\HS)$ for every $1<q<\infty$. \\
{\rm (ii)} Let $\gamma_0\ge0$ and let $m_\lambda$ satisfy the following two conditions: \\
{\rm (c)} There exists $\eta\in(0,\pi/2-\eps)$ such that for each $x_N>0$ and $\gamma\ge\gamma_0$, 
\[\tilde{\Sigma}_\eta^N \ni (\tau, \xi') \mapsto m_\lambda(\xi',x_N) \in\BC\]
 is bounded and holomorphic. \\
{\rm (d)} There exist $\eta\in(0,\pi/2-\eps)$ and $C>0$ such that $\sup\{|m_\lambda(\xi', x_N)| \mid (\tau, \xi')\in\tilde{\Sigma}_\eta^N\}\le Cx_N^{-1}$ for all $\gamma\ge\gamma_0$ and $x_N>0$. \\
Then $\tilde{T}_\gamma[m_\lambda]$ satisfies  
\[\|e^{-\gamma t}\tilde{T}_\gamma g\|_{L_p(\BR, L_q(\HS))} \le C \|e^{-\gamma t} g\|_{L_p(\BR, L_q(\HS))}\]
for every $\gamma\ge\gamma_0$ and $1<p,q<\infty$. 
\end{thm}

\begin{proof}
{\rm (i)} From the assumptions, equivalence of uniformly boundedness and $\CR$-boundedness on Hilbert space $\BC$ and Theorem \ref{H^infty}, we have 
\[\CR(\{{\xi'}^\alpha \pd_{\xi'}^\alpha m(\xi', x_N)\mid \xi'\in\dBR^{N-1}, \alpha\in\{0,1\}^{N-1}\})\le C_\eta x_N^{-1}.\]
This means that for each $x_N>0$, 
\begin{align*}
\|T[m]f(\cdot, x_N)\|_{L_q(\BR^{N-1})} &\le \int_0^\infty \|\CF_{\xi'}^{-1}m(\xi',x_N+y_N)\CF_{x'} f(\cdot, y_N)\|_{L_q(\BR^{N-1})} dy_N\\
&\le C\int_0^\infty \frac{\|f(\cdot, y_N)\|_{L_q(\BR^{N-1})}}{x_N+y_N} dy_N
\end{align*}
from Fourier multiplier theorem. 
And then, by Lemma \ref{kernel}, 
\[\|T[m]f\|_{L_q(\HS)} \le C\|f\|_{L_q(\HS)}\]
for some $C>0$. \\
{\rm (ii)} From the assumptions, equivalence of uniformly boundedness and $\CR$-boundedness on Hilbert space $\BC\in \CH\CT(\alpha)$, and Theorem \ref{H^infty}, we have 
\begin{align*}
\CR(\{{\xi'}^\alpha \pd_{\xi'}^\alpha (\tau \pd_\tau)^\beta m_\lambda(\xi', x_N)\mid \xi'\in\dBR^{N-1}, \tau\in \dBR, \alpha\in\{0,1\}^{N-1}, \beta\in\{0,1\}\})&\le C_\eta x_N^{-1}, 
\end{align*}
where $\gamma\ge\gamma_0$. 
This means that for each $x_N>0, \gamma\ge\gamma_0$, 
\[\{(\tau \pd_\tau)^\beta \CF^{-1}_{\xi'}m_\lambda(\xi',x_N)\CF_{x'}\mid \tau\in\dBR, \beta\in\{0,1\}\} \subset \CL(L_q(\BR^{N-1}))\]
is $\CR$-bounded and its $\CR$-norm is less than $Cx_N^{-1}$ by theorem \ref{Rbound}. 
Combining theorem \ref{Rkernel} with $D=(0,\infty), X=Y=L_q(\BR^{N-1}), k_0(x_N, y_N)=(x_N+y_N)^{-1}$ and Lemma \ref{kernel}, we have 
\begin{align*}
\{(\tau\pd_\tau)^\beta T[m_\lambda] \mid \tau\in \dBR, \beta\in\{0,1\}\}\subset \CL(L_q(\HS))
\end{align*}
is $\CR$-bounded. 
We use Fourier multiplier theorem again to get $\CF^{-1}_{\tau\to t}T[m_\lambda]\CF_{t\to\tau}$ is a bounded linear operator on $L_p(\BR, L_q(\HS))$, which conclude 
\[\|e^{-\gamma t}\tilde{T}_\gamma g\|_{L_p(\BR, L_q(\HS))} \le C \|e^{-\gamma t} g\|_{L_p(\BR, L_q(\HS))}\]
for every $\gamma\ge\gamma_0$ and $1<p,q<\infty$. 
\end{proof}

\begin{rem}
By the same proof, we are able to use anisotropic Lebesgue spaces between tangential direction and normal direction in the sense that 
\begin{align*}
\|T[m]f\|_{L_{q_1}(0,\infty, L_{q_2}(\BR^{N-1}))} &\le C\|f\|_{L_{q_1}(0,\infty, L_{q_2}(\BR^{N-1}))}, \\
\|e^{-\gamma t}\tilde{T}_\gamma g\|_{L_p(\BR, L_{q_1}(0,\infty, L_{q_2}(\BR^{N-1})))} &\le C \|e^{-\gamma t} g\|_{L_p(\BR, L_{q_1}(0,\infty, L_{q_2}(\BR^{N-1})))}
\end{align*}
with $1<p,q_1, q_2<\infty$ and $\|f\|_{L_{q_1}(0,\infty, L_{q_2}(\BR^{N-1}))}=(\int_0^\infty (\int_{\BR^{N-1}} |f(x',x_N)|^{q_2}dx')^{q_1/q_2}dx_N)^{1/q_1}$. 
\end{rem}

\subsection{Dirichlet boundary}
In section \ref{formula}, we obtained the solution formulas for Dirichlet boundary condition. 
We use the following identity; 
\[B^2=\lambda+\sum_{m=1}^{N-1} \xi_m^2, \qquad 1=\frac{B^2}{B^2}=\frac{\lambda^{1/2}}{B^2}\lambda^{1/2} - \sum_{m=1}^{N-1}\frac{i\xi_m}{B^2}(i\xi_m). \]

We decompose the solution operator so that the independent variables become the right-hand side of the estimates; 
\begin{align*}
u_j(x)&=- \sum_{k=1}^N \left\{\int_0^\infty  \CF_{\xi'}^{-1}\left[B^{-2}\pd_N\phi_{j,k}^D(\lambda, \xi', x_N+y_N)\CF_{x'} ((\lambda - \Delta')h_k)\right](x, y_N) dy_N \right.\\
&\qquad \qquad  + \int_0^\infty \CF_{\xi'}^{-1} \left[\lambda^{1/2}B^{-2}\phi_{j,k}^D(\lambda, \xi', x_N+y_N)\CF_{x'} (\lambda^{1/2}\pd_N h_k)\right](x,y_N) dy_N\\
& \qquad \qquad \left. - \sum_{m=1}^{N-1} \int_0^\infty \CF_{\xi'}^{-1} \left[i\xi_m B^{-2}\phi_{j,k}^D(\lambda, \xi', x_N+y_N)\CF_{x'} (\pd_m \pd_N h_k)\right](x,y_N) dy_N\right\}  \quad (j=1, \ldots, N),\\
\pi(x)&=- \sum_{k=1}^N \left\{\int_0^\infty  \CF_{\xi'}^{-1}\left[B^{-2}\pd_N\chi_k^D(\lambda, \xi', x_N+y_N)\CF_{x'} ((\lambda - \Delta')h_k)\right](x, y_N) dy_N \right.\\
&\qquad \qquad  + \int_0^\infty \CF_{\xi'}^{-1} \left[\lambda^{1/2}B^{-2}\chi_k^D(\lambda, \xi', x_N+y_N)\CF_{x'} (\lambda^{1/2}\pd_N h_k)\right](x,y_N) dy_N\\
& \qquad \qquad \left. - \sum_{m=1}^{N-1} \int_0^\infty \CF_{\xi'}^{-1} \left[i\xi_m B^{-2}\chi_k^D(\lambda, \xi', x_N+y_N)\CF_{x'} (\pd_m \pd_N h_k)\right](x,y_N) dy_N\right\}. 
\end{align*}

Let $S_u^D(\lambda, \xi', x_N)$ and $S_\pi^D(\lambda, \xi', x_N)$ be any of symbols; 
\begin{align*}
S_u^D(\lambda, \xi', x_N):=\begin{cases} B^{-2}\pd_N\phi_{j,k}^D(\lambda, \xi', x_N) & \text{or}, \\ \lambda^{1/2}B^{-2}\phi_{j,k}^D(\lambda, \xi', x_N)&\text{or}, \\i\xi_m B^{-2}\phi_{j,k}^D(\lambda, \xi', x_N), &\end{cases} \\
S_\pi^D(\lambda, \xi', x_N):=\begin{cases} B^{-2}\pd_N\chi_k^D(\lambda, \xi', x_N)&\text{or}, \\ \lambda^{1/2}B^{-2}\chi_k^D(\lambda, \xi', x_N)~(k\neq N)& \text{or}, \\ i\xi_m B^{-2}\chi_k^D(\lambda, \xi', x_N). &\end{cases}
\end{align*}
We are able to prove that all of the symbols are bounded in the sense that 
\begin{align}
&\sup_{\substack{(\lambda, \xi')\in \Sigma_\eps\times \tilde{\Sigma}_\eta^{N-1}\\ \ell, \ell'=1, \ldots, N-1}}\left\{(|\lambda|+|\lambda|^{1/2}|\xi_\ell|+|\xi_\ell| |\xi_{\ell'}|)|S_u^D|+(|\lambda|^{1/2}+|\xi_\ell|)|\pd_N S_u^D| + |\pd_N^2 S_u^D| + |\xi_\ell| |S_\pi^D| + |\pd_N S_\pi^D|\right\} \nonumber\\
&<Cx_N^{-1}\label{symbol}
\end{align}
for any $0<\eps<\pi/2$ and $0<\eta < \min\{\pi/4, \eps\}$ because of the identity 
\begin{align*}
\pd_N \CM_\lambda(\xi',x_N) &= - e^{-Bx_N} - A\CM_\lambda(\xi',x_N), \\
\pd^2_N  \CM_\lambda(\xi', x_N) &= (A+B)e^{-Bx_N} + A^2 \CM_\lambda(\xi', x_N), \\
\pd^3_N \CM_\lambda(\xi',x_N)&=-(A^2+AB+B^2)e^{-Bx_N} -A^3\CM_\lambda(\xi', x_N)
\end{align*}
and the estimate essentially given by Shibata--Shimizu \cite[Lemma 5.3]{SS12}; 
\begin{lem}\label{es}
Let $0<\eps<\pi/2$, $0<\eta<\min\{\pi/4, \eps/2\}$ and $m=0,1,2,3$. 
Then for any $(\lambda, \xi',x_N) \in \Sigma_\eps\times \tilde{\Sigma}_\eta^{N-1}\times(0,\infty)$, letting $A:=\sqrt{\sum_{j=1}^{N-1}\xi_j^2}$, $B:=\sqrt{\lambda + A^2}$ and $\tA:=\sqrt{\sum_{j=1}^{N-1}|\xi_j|^2}$, we have 
\begin{align}
c\tA & \le  \Re A \le |A| \le \tA, \tag{a}\\
c(|\lambda|^{1/2}+\tA)&\le \Re B \le |B| \le |\lambda|^{1/2} + \tA, \tag{b}\\
|\pd_N^m e^{-Bx_N}| &\le  (|\lambda|^{1/2}+\tA)^m e^{-c(|\lambda|^{1/2}+\tA)x_N} \le C (|\lambda|^{1/2} + \tA)^{-1+m}x_N^{-1}, \tag{c}\\
|\pd_N^m e^{-Ax_N}| & \le \tA^m e^{-c\tA x_N} \le  C\tA^{-1+m} x_N^{-1}, \tag{d}\\
|\CM_\lambda(\xi', x_N)| &\le C (x_N~{\rm or}~|\lambda|^{-1/2})e^{-c\tA x_N} \le C(\tA^{-2}~{\rm or}~|\lambda|^{-1/2}\tA^{-1}) x_N^{-1}, \tag{e}\\
|\CM_\lambda(\xi', x_N)| &\le C |\lambda|^{-1}x_N^{-1}, \tag{f}\\
|\pd_N^m \CM_\lambda(\xi', x_N)|& \le C(|\lambda|^{1/2}+\tA)^{-2+m}x_N^{-1}, \tag{g}
\end{align}
with positive constants $c$ and $C$, which are independent of $\lambda, \xi', x_N$. 
\end{lem}
This lemma is proved in \ref{appA}. 
We remark that the paper \cite{SS12} treated for $\xi'\in\mathbb{R}^{N-1}\setminus\{0\}$ although above theorem is $\xi'\in\tilde{\Sigma}_\eta^{N-1}$. 
Since we prepare theorem \ref{thm}, we do not need the estimate of derivatives of the symbols. 
For the excluded term $\lambda^{1/2}B^{-2}\chi_N^D(\lambda, \xi', x_N)$, we see that 
\begin{align*}
&\int_0^\infty \CF_{\xi'}^{-1} \left[\lambda^{1/2}B^{-2}\chi_N^D(\lambda, \xi', x_N+y_N)\CF_{x'} (\lambda^{1/2}\pd_N h_N)\right](x,y_N) dy_N\\
=&\int_0^\infty \CF_{\xi'}^{-1} \left[A B^{-2}\chi_N^D(\lambda, \xi', x_N+y_N)\CF_{x'} (\lambda |\nabla'|^{-1}\pd_N h_N)\right](x,y_N) dy_N
\end{align*}
and 
\begin{align*}
\sup_{\substack{(\lambda, \xi')\in \Sigma_\eps\times \tilde{\Sigma}_\eta^{N-1}\\ \ell=1, \ldots, N-1}}\left\{|\xi_\ell| |A B^{-2}\chi_N^D(\lambda, \xi', x_N)| + |\pd_N (A B^{-2}\chi_N^D(\lambda, \xi', x_N))|\right\} <Cx_N^{-1}. 
\end{align*}

The inequality \eqref{symbol} with above discussion corresponds to the estimates $|\lambda| u$, $|\lambda|^{1/2}\pd_\ell u$, $\pd_\ell \pd_{\ell'} u$, $|\lambda|^{1/2}\pd_N u$, $\pd_\ell \pd_N u$, $\pd_N^2 u$, and $\pd_\ell \pi$ and $\pd_N\pi$ respectively. 

We also see that the new symbols $S_u^D$ and $S_\pi^D$, multiplied $\lambda$, $\xi_\ell$ and $\pd_N$, are holomorphic in $(\tau, \xi')\in\tilde{\Sigma}_\eta^N$. 
Therefore we are able to use theorem \ref{thm}. 

\begin{thm}
Let $0<\eps<\pi/2$ and $1<q<\infty$. 
Then for any $\lambda \in \Sigma_\eps, h\in W^2_q(\HS)$ and $h_N\in E_q(\HS)$, problem \eqref{resolvent Stokes} with Dirichlet boundary condition and $f=g=0$ admits a solution $(u, \pi) \in W^2_q(\HS) \times \hW^1_q(\HS)$ with the resolvent estimate; 
\begin{align*}
\|(|\lambda|u, |\lambda|^{1/2}\nabla u, \nabla^2 u, \nabla \pi)\|_{L_q(\HS)} 
&\le
C\|(|\lambda|h, |\lambda|^{1/2} \nabla h, \nabla^2 h, |\lambda| |\nabla'|^{-1}\pd_N h_N) \|_{L_q(\HS)}\\
&\le
C\|(|\lambda|h, \nabla^2 h, |\lambda| |\nabla'|^{-1}\pd_N h_N) \|_{L_q(\HS)}
\end{align*}
for some constant $C=C_{N, q, \eps}$ depending only on $N, q$ and $\eps$. 
\end{thm}

This theorem and the estimates in section \ref{reduction} derives the existence part of theorem \ref{resolventthm} with Dirichlet boundary condition. 
The uniqueness was proved in \cite[p.121]{KS12} where they considered the homogeneous equation and the dual problem. 

For the non-stationary Stokes equations we have, by theorem \ref{thm} again,  
\begin{thm}
Let $1<p, q<\infty$ and $\gamma_0\ge 0$. 
Then for any 
\[H\in W^1_{p,0,\gamma_0}(\BR, L_q(\HS)) \cap L_{p,0,\gamma_0}(\BR, W^2_q(\HS)), \quad H_N \in H^1_{p, 0,\gamma_0}(\BR, E_q(\HS))\]
problem \eqref{non-stationary Stokes} with Dirichlet boundary condition, $F=G=0$ and time interval $\BR$ admits a solution $(U, \Pi)$ such that 
\begin{align*}
U&\in W^1_{p,0,\gamma_0}(\BR, L_q(\HS)) \cap L_{p,0,\gamma_0}(\BR, W^2_q(\HS)), \\
\Pi &\in L_{p,0,\gamma_0}(\BR, \hW^1_q(\HS))
\end{align*} 
with the maximal $L_p$-$L_q$ regularity; 
\begin{align*}
&\|e^{-\gamma t}(\pd_t U, \gamma U, \Lambda^{1/2}_\gamma \nabla U, \nabla^2 U, \nabla \Pi)\|_{L_p(\BR, L_q(\HS))}\\
\le&
C\|e^{-\gamma t} (\pd_t H, \Lambda^{1/2}_\gamma \nabla H, \nabla^2 H, \pd_t(|\nabla'|^{-1}\pd_N H_N) \|_{L_p(\BR, L_q(\HS))}\\
\le& 
C\|e^{-\gamma t} (\pd_t H, \nabla^2 H, \pd_t(|\nabla'|^{-1}\pd_N H_N)) \|_{L_p(\BR, L_q(\HS))}
\end{align*}
for any $\gamma \ge \gamma_0$ with some constant $C=C_{N, p, q, \gamma_0}$ depending only on $N, p, q$ and $\gamma_0$. 
\end{thm}

\begin{proof}
Almost all of the proof has already done. 
We need to prove $U=\nabla \Pi =0$ for $t<0$ and 
\[\|e^{-\gamma t} \gamma U\|_{L_p(\BR, L_q(\HS))}\le C\|e^{-\gamma t} (\pd_t H, \Lambda^{1/2}_\gamma \nabla H, \nabla^2 H) \|_{L_p(\BR, L_q(\HS))}.\]
This is easily proved from $\|\gamma u\|_{L_q(\HS)} \le \||\lambda|u \|_{L_q(\HS)}$ in the resolvent estimates. 
Vanishing property is same as \cite{KS12, SS11}. 
\end{proof}

\subsection{Neumann boundary}
Using the result in section \ref{formula}, we have the following form
\begin{align*}
u_j(x)&=- \sum_{k=1}^N \left\{\int_0^\infty  \CF_{\xi'}^{-1}\left[\lambda^{1/2}B^{-2}\pd_N\phi_{j,k}^N(\lambda, \xi', x_N+y_N)\CF_{x'} (\lambda^{1/2}h_k)\right](x, y_N) dy_N \right.\\
&\qquad \qquad  - \sum_{m=1}^{N-1}\int_0^\infty \CF_{\xi'}^{-1} \left[i\xi_mB^{-2} \pd_N\phi_{j,k}^N(\lambda, \xi', x_N+y_N)\CF_{x'} (\pd_m h_k)\right](x,y_N) dy_N\\
&\qquad \qquad \left. + \int_0^\infty \CF_{\xi'}^{-1} \left[\phi_{j,k}^N(\lambda, \xi', x_N+y_N)\CF_{x'} (\pd_N h_k)\right](x,y_N) dy_N\right\} \quad (j=1, \ldots, N),\\
\pi(x)&=- \sum_{k=1}^N \left\{\int_0^\infty  \CF_{\xi'}^{-1}\left[\lambda^{1/2}B^{-2}\pd_N\chi_k^N(\lambda, \xi', x_N+y_N)\CF_{x'} (\lambda^{1/2}h_k)\right](x, y_N) dy_N \right.\\
&\qquad \qquad  - \sum_{m=1}^{N-1} \int_0^\infty \CF_{\xi'}^{-1} \left[i\xi_m \pd_N \chi_k^N(\lambda, \xi', x_N+y_N)\CF_{x'} (\pd_m h_k)\right](x,y_N) dy_N\\
&\qquad \qquad \left. + \int_0^\infty \CF_{\xi'}^{-1} \left[\chi_k^N(\lambda, \xi', x_N+y_N)\CF_{x'} (\pd_N h_k)\right](x,y_N) dy_N\right\}. 
\end{align*}

Let $S_u^N(\lambda, \xi', x_N)$ and $S_\pi^N(\lambda, \xi', x_N)$ be any of symbols; 
\begin{align*}
S_u^N(\lambda, \xi', x_N):=\begin{cases} \lambda^{1/2}B^{-2}\pd_N\phi_{j,k}^N(\lambda, \xi', x_N) & \text{or}, \\ i\xi_m B^{-2}\pd_N\phi_{j,k}^N(\lambda, \xi', x_N)&\text{or}, \\ \phi_{j,k}^N(\lambda, \xi', x_N), &\end{cases} \\
S_\pi^N(\lambda, \xi', x_N):=\begin{cases} \lambda^{1/2} B^{-2}\pd_N\chi_k^N(\lambda, \xi', x_N)&\text{or}, \\ i\xi_m B^{-2}\pd_N \chi_k^N(\lambda, \xi', x_N) & \text{or}, \\ \chi_k^N(\lambda, \xi', x_N). &\end{cases}
\end{align*}
We are able to prove that all of the symbols are bounded in the sense that 
\begin{align*}
&\sup_{\substack{(\lambda, \xi')\in \Sigma_\eps\times \tilde{\Sigma}_\eta^{N-1}\\ \ell, \ell'=1, \ldots, N-1}}\left\{(|\lambda|+|\lambda|^{1/2}|\xi_\ell|+|\xi_\ell| |\xi_{\ell'}|)|S_u^N|+(|\lambda|^{1/2}+|\xi_\ell|)|\pd_N S_u^N| + |\pd_N^2 S_u^N| + |\xi_\ell| |S_\pi^N| + |\pd_N S_\pi^N|\right\} \nonumber\\
&<Cx_N^{-1}
\end{align*}
by the estimates in lamma \ref{es} and the following lemma. 

\begin{lem}\label{es2}
Let $0<\eps< \pi/2$ and $0<\eta < \min\{\pi/4, \eps/2\}$. 
Then there exists a positive constant $c$ such that 
\[ c(|\lambda|^{1/2} + \tA)^3 \le |D(A,B)|\qquad (\lambda\in \Sigma_\eps, \xi'\in \tilde{\Sigma}_\eta^{N-1}), \]
where $D(A,B)= B^3 + AB^2 + 3AB^2 -A^3$. 
\end{lem}
The proof is given in \ref{appB}. 
This is a generalization of \cite[Lemma 4.4]{SS03} in which they proved for $\xi'\in\BR^{N-1}\setminus\{0\}$. 

Since the new symbols are holomorphic in $(\tau, \xi')\in\tilde{\Sigma}_\eta^N$, we apply theorem \ref{thm} for Neumann boundary condition. 

\begin{thm}
Let $0<\eps<\pi/2$ and $1<q<\infty$. 
Then for any $\lambda \in \Sigma_\eps, h\in W^1_q(\HS)$, problem \eqref{resolvent Stokes} with Neumann condition and $f=g=0$ admits a solution $(u, \pi) \in W^2_q(\HS) \times \hW^1_q(\HS)$ with the resolvent estimate; 
\begin{align*}
\|(|\lambda|u, |\lambda|^{1/2}\nabla u, \nabla^2 u, \nabla \pi)\|_{L_q(\HS)} 
\le
C\|(|\lambda|^{1/2}h, \nabla h) \|_{L_q(\HS)}
\end{align*}
for some constant $C=C_{N, q, \eps}$ depending only on $N, q$ and $\eps$. 
\end{thm}

The uniqueness is proved in \cite{SS12}. 

For the non-stationary Stokes equations we have 
\begin{thm}
Let $1<p, q<\infty$ and $\gamma_0\ge0$. 
Then for any 
\[H\in H^{1/2}_{p,0,\gamma_0}(\BR, L_q(\HS)) \cap L_{p,0,\gamma_0}(\BR, W^1_q(\HS))\]
problem \eqref{non-stationary Stokes} with Neumann boundary condition, $F=G=0$ and time interval $\BR$ admits a solution $(U, \Pi)$ such that 
\begin{align*}
U&\in W^1_{p,0,\gamma_0}(\BR, L_q(\HS)) \cap L_{p,0,\gamma_0}(\BR, W^2_q(\HS)), \\
\Pi &\in L_{p,0,\gamma_0}(\BR, \hW^1_q(\HS))
\end{align*} 
with the maximal $L_p$-$L_q$ regularity; 
\begin{align*}
\|e^{-\gamma t}(\pd_t U, \gamma U, \Lambda^{1/2}_\gamma \nabla U, \nabla^2 U, \nabla \Pi)\|_{L_p(\BR, L_q(\HS))}
\le
C\|e^{-\gamma t} (\Lambda^{1/2}_\gamma H, \nabla H) \|_{L_p(\BR, L_q(\HS))}
\end{align*}
for any $\gamma \ge \gamma_0$ with some constant $C=C_{N, p, q, \gamma_0}$ depending only on $N, p, q$ and $\gamma_0$. 
\end{thm}

\subsection{Robin boundary}
Using the result in section \ref{formula}, we decompose as follows; 
\begin{align*}
u_j(x)&=- \sum_{k=1}^{N-1} \left\{\int_0^\infty  \CF_{\xi'}^{-1}\left[\lambda^{1/2}B^{-2}\pd_N\phi_{j,k}^R(\lambda, \xi', x_N+y_N)\CF_{\xi'} (\lambda^{1/2}h_k)\right](x, y_N) dy_N \right.\\
&\qquad \qquad  - \sum_{m=1}^{N-1}\int_0^\infty \CF_{\xi'}^{-1} \left[i\xi_mB^{-2}\pd_N \phi_{j,k}^R(\lambda, \xi', x_N+y_N)\CF_{x'} (\pd_m h_k)\right](x,y_N) dy_N\\
&\qquad \qquad \left. + \int_0^\infty \CF_{\xi'}^{-1} \left[\phi_{j,k}^R(\lambda, \xi', x_N+y_N)\CF_{x'} (\pd_N h_k)\right](x,y_N) dy_N\right\} \\
&\qquad \qquad  - \left\{\int_0^\infty  \CF_{\xi'}^{-1}\left[(B^{-2}\pd_N\phi_{j,N}^R(\lambda, \xi', x_N+y_N))\CF_{x'} ((\lambda - \Delta')h_N)\right](x, y_N) dy_N \right.\\
&\qquad \qquad  + \int_0^\infty \CF_{\xi'}^{-1} \left[\lambda^{1/2}B^{-2}\phi_{j,N}^R(\lambda, \xi', x_N+y_N)\CF_{x'} (\lambda^{1/2}\pd_N h_N)\right](x,y_N) dy_N\\
& \qquad \qquad \left. - \sum_{m=1}^{N-1} \int_0^\infty \CF_{\xi'}^{-1} \left[i\xi_m B^{-2}\phi_{j,N}^R(\lambda, \xi', x_N+y_N)\CF_{x'} (\pd_m \pd_N h_N)\right](x,y_N) dy_N\right\}\quad (j=1, \ldots, N),\\
\pi(x)&=- \sum_{k=1}^{N-1} \left\{\int_0^\infty  \CF_{\xi'}^{-1}\left[\lambda^{1/2}B^{-2}\pd_N\chi_k^R(\lambda, \xi', x_N+y_N)\CF_{x'} (\lambda^{1/2}h_k)\right](x, y_N) dy_N \right.\\
&\qquad \qquad  - \sum_{m=1}^{N-1}\int_0^\infty \CF_{\xi'}^{-1} \left[i\xi_mB^{-2}\pd_N \chi_k^R(\lambda, \xi', x_N+y_N)\CF_{x'} (\pd_m h_k)\right](x,y_N) dy_N\\
&\qquad \qquad \left. + \int_0^\infty \CF_{\xi'}^{-1} \left[\chi_k^R(\lambda, \xi', x_N+y_N)\CF_{x'} (\pd_N h_k)\right](x,y_N) dy_N\right\} \\
&\qquad \qquad - \left\{\int_0^\infty  \CF_{\xi'}^{-1}\left[(B^{-2}\pd_N\chi_N^R(\lambda, \xi', x_N+y_N))\CF_{x'} ((\lambda - \Delta')h_N)\right](x, y_N) dy_N \right.\\
&\qquad \qquad  + \int_0^\infty \CF_{\xi'}^{-1} \left[\lambda^{1/2}B^{-2}\chi_N^R(\lambda, \xi', x_N+y_N)\CF_{x'} (\lambda^{1/2}\pd_N h_N)\right](x,y_N) dy_N\\
& \qquad \qquad \left. - \sum_{m=1}^{N-1} \int_0^\infty \CF_{\xi'}^{-1} \left[i\xi_m B^{-2}\chi_N^R(\lambda, \xi', x_N+y_N)\CF_{x'} (\pd_m \pd_N h_N)\right](x,y_N) dy_N\right\}. 
\end{align*}

Let $S_u^R(\lambda, \xi', x_N)$ and $S_\pi^R(\lambda, \xi', x_N)$ be any of symbols; 
\begin{align*}
S_u^R(\lambda, \xi', x_N):=\begin{cases} \lambda^{1/2}B^{-2}\pd_N\phi_{j,k}^R(\lambda, \xi', x_N) & \text{or}, \\ i\xi_m B^{-2}\pd_N\phi_{j,k}^R(\lambda, \xi', x_N)&\text{or}, \\ \phi_{j,k}^R(\lambda, \xi', x_N), &\text{or}, \\B^{-2}\pd_N\phi_{j,N}^R(\lambda, \xi', x_N) & \text{or}, \\ \lambda^{1/2}B^{-2}\phi_{j,N}^R(\lambda, \xi', x_N)&\text{or}, \\i\xi_m B^{-2}\phi_{j,N}^R(\lambda, \xi', x_N), &
\end{cases} \\
S_\pi^R(\lambda, \xi', x_N):=\begin{cases} \lambda^{1/2} B^{-2}\pd_N\chi_k^R(\lambda, \xi', x_N)&\text{or}, \\ i\xi_m B^{-2}\pd_N \chi_k^R(\lambda, \xi', x_N) & \text{or}, \\ \chi_k^R(\lambda, \xi', x_N), &\text{or}, \\B^{-2}\pd_N\chi_N^R(\lambda, \xi', x_N) & \text{or}, \\ i\xi_m B^{-2}\chi_N^R(\lambda, \xi', x_N), &\end{cases}
\end{align*}
where $k=1, \ldots, N-1$ and it is different from Neumann boundary. 

We are able to prove that all of the symbols are bounded in the sense that 
\begin{align*}
&\sup_{\substack{(\lambda, \xi')\in \Sigma_\eps\times \tilde{\Sigma}_\eta^{N-1}\\ \ell, \ell'=1, \ldots, N-1}}\left\{(|\lambda|+|\lambda|^{1/2}|\xi_\ell|+|\xi_\ell| |\xi_{\ell'}|)|S_u^R|+(|\lambda|^{1/2}+|\xi_\ell|)|\pd_N S_u^R| + |\pd_N^2 S_u^R| + |\xi_\ell| |S_\pi^R| + |\pd_N S_\pi^R|\right\} \\
&<Cx_N^{-1}. 
\end{align*}

For the excluded term $\lambda^{1/2}B^{-2}\chi_N^R(\lambda, \xi', x_N)$, we see that 
\begin{align*}
&\int_0^\infty \CF_{\xi'}^{-1} \left[\lambda^{1/2}B^{-2}\chi_N^R(\lambda, \xi', x_N+y_N)\CF_{x'} (\lambda^{1/2}\pd_N h_N)\right](x,y_N) dy_N\\
=&\int_0^\infty \CF_{\xi'}^{-1} \left[A B^{-2}\chi_N^R(\lambda, \xi', x_N+y_N)\CF_{x'} (\lambda |\nabla'|^{-1}\pd_N h_N)\right](x,y_N) dy_N
\end{align*}
and 
\begin{align*}
\sup_{\substack{(\lambda, \xi')\in \Sigma_\eps\times \tilde{\Sigma}_\eta^{N-1}\\ \ell=1, \ldots, N-1}}\left\{|\xi_\ell| |A B^{-2}\chi_N^R(\lambda, \xi', x_N)| + |\pd_N (A B^{-2}\chi_N^R(\lambda, \xi', x_N))|\right\} <Cx_N^{-1}. 
\end{align*}

Since the new symbols are holomorphic in $(\tau, \xi')\in\tilde{\Sigma}_\eta^N$, we apply theorem \ref{thm} for Robin boundary condition. 

\begin{thm}
Let $0<\eps<\pi/2$ and $1<q<\infty$. 
Then for any $\lambda \in \Sigma_\eps, h'\in W^1_q(\HS)$ and $h_N\in E_q(\HS)$, problem \eqref{resolvent Stokes} with Robin condition and $f=g=0$ admits a solution $(u, \pi) \in W^2_q(\HS) \times \hW^1_q(\HS)$ with the resolvent estimate; 
\begin{align*}
\|(|\lambda|u, |\lambda|^{1/2}\nabla u, \nabla^2 u, \nabla \pi)\|_{L_q(\HS)} 
\le
C\|(|\lambda|^{1/2}h', \nabla h', |\lambda|h_N, \nabla^2 h_N, |\lambda||\nabla'|^{-1}\pd_N h_N) \|_{L_q(\HS)}
\end{align*}
for some constant $C=C_{N, q, \eps, \alpha, \beta}$ depending only on $N, q, \eps,\alpha$ and $\beta$. 
\end{thm}

The uniqueness is proved in \cite{SS07}. 

For the non-stationary Stokes equations we have 
\begin{thm}
Let $1<p, q<\infty$ and $\gamma_0\ge0$. 
Then for any 
\begin{align*}
&H'\in H^{1/2}_{p,0,\gamma_0}(\BR, L_q(\HS)) \cap L_{p,0,\gamma_0}(\BR, W^1_q(\HS)), \\
&H_N\in W^1_{p,0,\gamma_0}(\BR, L_q(\HS)) \cap L_{p,0,\gamma_0}(\BR, W^2_q(\HS)) \cap H^1_{p,0,\gamma_0}(\BR, E_q(\HS)), 
\end{align*}
problem \eqref{non-stationary Stokes} with Robin boundary condition, $F=G=0$ and time interval $\BR$ admits a solution $(U, \Pi)$ such that 
\begin{align*}
U&\in W^1_{p,0,\gamma_0}(\BR, L_q(\HS)) \cap L_{p,0,\gamma_0}(\BR, W^2_q(\HS)), \\
\Pi &\in L_{p,0,\gamma_0}(\BR, \hW^1_q(\HS))
\end{align*} 
with the maximal $L_p$-$L_q$ regularity; 
\begin{align*}
&\|e^{-\gamma t}(\pd_t U, \gamma U, \Lambda^{1/2}_\gamma \nabla U, \nabla^2 U, \nabla \Pi)\|_{L_p(\BR, L_q(\HS))}\\
\le&
C\|e^{-\gamma t} (\Lambda^{1/2}_\gamma H', \nabla H', \pd_t H_N, \nabla^2 H_N, \pd_t(|\nabla'|^{-1}\pd_N H_N) \|_{L_p(\BR, L_q(\HS))}
\end{align*}
for any $\gamma \ge \gamma_0$ with some constant $C=C_{N, p, q, \alpha, \beta, \gamma_0}$ depending only on $N, p, q, \alpha, \beta$ and $\gamma_0$. 
\end{thm}

\appendix
\def\thesection{Appendix \Alph{section}}
\section{Proof of the estimate for normal component}\label{appX}
\begin{proof}
We see that 
\begin{align*}
|\lambda||\nabla'|^{-1} \pd_N v_N = \sum_{k=1}^{N-1}\CF_{\xi}^{-1} \left( |\lambda| \frac{i\xi_N}{|\xi'|}\frac{1}{\lambda+|\xi|^2}(\frac{-\xi_N\xi_k}{|\xi|^2})\right) \CF_x f_k^o + \CF_{\xi}^{-1}\left( |\lambda| \frac{i\xi_N}{|\xi'|}\frac{1}{\lambda+|\xi|^2}( 1- \frac{\xi_N^2}{|\xi|^2})\right)\CF_x f_N^e. 
\end{align*}
All symbols 
\begin{align*}
|\lambda|\frac{i\xi_N}{|\xi'|}\frac{1}{\lambda+|\xi|^2}\frac{-\xi_N\xi_k}{|\xi|^2}, \quad |\lambda|\frac{i\xi_N}{|\xi'|}\frac{1}{\lambda+|\xi|^2}(1-\frac{\xi_N^2}{|\xi|^2})=|\lambda|\frac{i\xi_N}{|\xi'|}\frac{1}{\lambda+|\xi|^2}\frac{|\xi'|^2}{|\xi|^2}
\end{align*}
are bounded and holomorphic in $\lambda\in \Sigma_\eps$, $\xi\in \tilde\Sigma_\eta^N$ for small $\eps, \eta$, where we regard $|\xi'|=\sqrt{\sum_{j=1}^{N-1} \xi_j^2}=A$ and $|\xi|^2=A^2+\xi_N^2$ as complex functions. 
Therefore, by theorem \ref{thm}, we have 
\begin{align*}
\||\lambda||\nabla'|^{-1} \pd_N v_N\|_{L_q(\HS)}\le \||\lambda||\nabla'|^{-1} \pd_N v_N\|_{L_q(\BR^N)} \le C(\sum_{k=1}^{N-1}\|f_k^o\|_{L_q(\BR^N)} + \|f_N^e\|_{L_q(\BR^N)})
\le C'\|f\|_{L_q(\HS)}. 
\end{align*}
The other estimate follows similarly. 
\end{proof}

\section{Proof of Lemma \ref{es}.}\label{appA}
\begin{proof}
Let $\lambda = re^{i\theta}(\in \Sigma_\eps)$ and $\xi'=(\xi_1, \ldots, \xi_{N-1})=(\alpha_1 e^{i\beta_1}, \ldots, \alpha_{N-1}e^{i\beta_{N-1}})(\in \tilde{\Sigma}_\eta^{N-1})$ . 
We see 
\begin{align*}
|\lambda+\xi_1^2|^2 &= r^2 + \alpha_1^4 + 2r \alpha_1^2\cos(\theta-2\beta_1)\\
&\ge r^2 + \alpha_1^4 -2r\alpha_1^2 \cos(\eps-2\eta)\\
&= (r-\alpha_1^2)^2 \cos(\eps-2\eta) + (1-\cos(\eps-\eta))(r^2+\alpha_1^4)\\
&\ge 2\sin^2\frac{\eps-2\eta}{2}(r^2+\alpha_1^4)\\
&\ge \sin^2\frac{\eps-2\eta}{2}(r+\alpha_1^2)^2. 
\end{align*}
Therefore $|\lambda+\xi_1^2|\ge \sin((\eps-2\eta)/2)(|\lambda|+|\xi_1|^2)$. 
Since $\lambda+\xi_1^2\in\Sigma_\eps$, we have  
\begin{align*}
|\lambda+\xi_1^2+ \xi_2^2 |\ge \sin\frac{\eps-2\eta}{2}(|\lambda+\xi_1^2|+|\xi_2|^2) \ge \sin^2\frac{\eps-2\eta}{2}(|\lambda|+|\xi_1|^2+|\xi_2|^2). 
\end{align*}
Inductively, $|\lambda+A^2| \ge \sin^{N-1}((\eps-2\eta)/2)(|\lambda|+\tA^2)$. 
Let $\tilde{\theta}=\arg(\lambda+A^2)$. 
Then $|\tilde{\theta}|<\pi-\eps$ and $\cos(\tilde{\theta}/2)> \sin(\eps/2)$, so 
\begin{align*}
\Re B = |\lambda+A^2|^{1/2}\cos\frac{\tilde{\theta}}{2} \ge \sin\frac{\eps}{2}(|\lambda|+\tA^2). 
\end{align*}
This proves (b) since others are obvious. 
Similar proof holds for $\lambda=0$, which means (a). 
Inequalities (a) and (b) derive inequalities (c) and (d) easily. 
Inequality (e) is same as \cite{SS12} thanks to (a) and (b). 
For (f), we see that 
\begin{align*}
|\CM_\lambda(\xi', x_N)| &\le C |\lambda|^{-1/2} e^{-c \tA x_N/2} e^{-|\lambda|^{1/2}x_N/2} \\
&\le C |\lambda|^{-1}x_N^{-1} e^{-c \tA x_N/2}, 
\end{align*}
where the first inequality is in \cite{SS12}. 
For (g), we derive by combining all results from (a) to (f) except for (d).  
\end{proof}

\section{Proof of Lemma \ref{es2}.}\label{appB}
\begin{proof} We notice that $D(A,B)\neq0$ for $(\lambda, \xi')\in\Sigma_\eps \times \tilde\Sigma_\eta^{N-1}$ is enough to prove the theorem since lower bound is same as in \cite{SS03}. 
Let $B/A=\sqrt{\lambda A^{-2}+1}=a+bi$ with $a, b\in\BR$, then $a>1$ from the assumption. 
We see 
\begin{align*}
f(A,B):=\frac{D(A,B)}{A^3} = (a^3+a^2-3a(b^2-1)-b^2-1) + i(3a^2+2a-b^2+3)b. 
\end{align*}
For $b=0$, then $\Im f=0$, but $\Re f = a^3+a^2+3a-1\neq0$ for $a>1$. 
For $b^2=3a^2+2a+3$, then $\Im f=0$, but $\Re f = -8a^3-8a^2-8a-4 \neq 0$ for $a>1$. 
This means $D(A,B)\neq 0$. 
\end{proof}

\subsection*{Acknowledgements} The authors thank Professor Hirokazu Saito for some remarks on this paper. 
The research was supported by JSPS KAKENHI Grant No.\,19K23408.

\end{document}